\documentclass[11pt]{amsart}
\usepackage{amsmath}
\usepackage{tikz}

\theoremstyle{plain}
\newtheorem{theorem}{Theorem}[section]
\newtheorem{lemma}[theorem]{Lemma}
\newtheorem{prop}[theorem]{Proposition}

\theoremstyle{definition}

\numberwithin{equation}{section}

\def\be{\begin{equation}}
\def\ee{\end{equation}}

\begin{document}

\title[Constant Mean Curvature Surfaces in the Hyperbolic Space]
{Boundary Expansions for Constant Mean Curvature Surfaces in the Hyperbolic Space}
\author{Qing Han}
\address{Department of Mathematics\\
University of Notre Dame\\
Notre Dame, IN 46556} \email{qhan@nd.edu}
\address{Beijing International Center for Mathematical Research\\
Peking University\\
Beijing, 100871, China} \email{qhan@math.pku.edu.cn}
\author{Yue Wang}
\address{School of Mathematical Sciences\\
Peking University\\
Beijing, 100871, China}
\address{Beijing International Center for Mathematical Research\\
Peking University\\
Beijing, 100871, China} \email{1201110027@pku.edu.cn}

\begin{abstract}
We study expansions near the boundary of solutions to the Dirichlet problem
for the constant mean curvature equation in the hyperbolic space. With a characterization of remainders
of the expansion by multiple integrals, we establish optimal asymptotic expansions
of solutions with boundary values of finite regularity and demonstrate a slight loss of
regularity for coefficients.
\end{abstract}

\thanks{The first author acknowledges the support of NSF
Grant DMS-1404596. The second author acknowledges the support of NSFC Grant NSFC11571019.}
\maketitle

\section{Introduction}\label{sec-Intro}

Complete minimal hypersurfaces in the hyperbolic space
$\mathbb H^{n+1}$ demonstrate similar properties as
those in the Euclidean space $\mathbb R^{n+1}$ in the aspect of the interior regularity and
different properties in the aspect of the boundary regularity.
Anderson \cite{Anderson1982Invent}, \cite{Anderson1983}
studied complete area-minimizing submanifolds
and proved that, for any
given closed embedded $(n-1)$-dimensional submanifold
$N$ at the infinity of $\mathbb H^{n+1}$,
there exists
a complete
area minimizing integral $n$-current which is asymptotic to
$N$ at infinity.
In the case $n\le 6$, these currents are
embedded smooth submanifolds;
while in the case $n\ge 7$, as in the Euclidean case, there can be closed singular set
of Hausdorff dimension at most $n-7$.
Hardt and
Lin \cite{Hardt&Lin1987}  discussed the $C^1$-boundary regularity of such hypersurfaces.
Subsequently,
Lin \cite{Lin1989Invent} studied the higher order boundary regularity.
In a more general setting, Graham and Witten \cite{Graham&Witten1999}
studied $n$-dimensional minimal surfaces
of any codimension in asymptotically hyperbolic manifolds
and derived an expansion of the normalized area up to order $n+1$. Recently, Han and Jiang
\cite{HanJiang} studied the asymptotic expansions of minimal surfaces in the hyperbolic space
and derived optimal estimates for the remainders in the context of the finite regularity.

It is natural to study constant mean curvature hypersurfaces in the hyperbolic space and investigate
the differences and similarities between the nonzero mean curvature case and the zero mean curvature case.
Tonegawa \cite{Tonegawa1996MathZ} discussed the constant mean curvature hypersurfaces in the hyperbolic space
and established several regularity results near boundary. In this paper, we will study expansions of the
constant mean curvature hypersurfaces in the hyperbolic space near boundary.

We first introduce the differential equation for constant mean curvature hypersurfaces in the hyperbolic space near
infinity. Denote by $x=(x', x_n)$ points in $\mathbb R^n$
and consider the function $u=u(x)$ in $B_1^+\subset\mathbb R^n$ satisfying
\begin{equation}\label{eq-Intro-Equation}
\Delta u-\frac{u_iu_j}{1+|Du|^2}u_{ij}-\frac{n}{x_n}(u_n-H\sqrt{1+|Du|^2})=0\quad\text{in }B_1^+,\end{equation}
and
\begin{equation}\label{eq-Intro-BoundaryValue}
u=\varphi\quad\text{on }B_1',\end{equation}
where $\varphi$ is a given function on $B_1'$ and $H$ a contant in $B_1^+$.
The graph of $u$ for $\{x_n>0\}$ is a hypersurface in $\mathbb H^{n+1}$ with its mean curvature given by $H$
and its asymptotic boundary given by the graph of $\varphi$.

In this paper, we discuss the boundary regularity of $u$
by expanding $u$ in terms of $x_n$.
We discuss a general case and allow $H$ to be a function in $\bar B_1^+$.
We can
write formal expansions for solutions of \eqref{eq-Intro-Equation}-\eqref{eq-Intro-BoundaryValue}
in the following form:
\begin{align*}
u=\sum_{i=0}^nc_ix^i+\sum_{i=n+1}^\infty
\sum_{j=0}^{N_i}c_{i,j} x_n^i (\log x_n)^j,
\end{align*}
where $c_i$ and $c_{i,j}$ are functions of $x'\in B_1'$
and $N_i$ is a nonnegative integer depending on $i$, with $N_{n+1}=1$.
A formal calculation can only determine {\it finitely many terms}
in the formal expansion of $u$.
In fact, the coefficients $c_0, c_1$, $\cdots, c_{n}$, and $c_{n+1,1}$
have explicit expressions in terms of $\varphi$ and $H$. For example, we have
$$c_0=\varphi,$$
and
$$c_{1}=H_0\sqrt{\frac{1+|D_{x'}\varphi|^2}{1-H_0^{2}}},$$
where $H_0=H(\cdot, 0)$. In this paper, we always assume $|H|<1$ in $\bar B_1^+$.
See Section \ref{sec-FormalExpansions} for expressions of other $c_i$. We note that
logarithmic terms usually appear in all dimensions, except in the dimension 2. In fact, for $n=2$ and
constant $H$,
$c_{i,j}=0$ for all $i\ge 3$ and $j\ge 1$. This was observed by
Tonegawa \cite{Tonegawa1996MathZ}.

We point out that the case of nonzero $H$ is different from the case $H\equiv0$. For $H\equiv0$, we have,
for $n$ even,
\begin{align*}
u=\varphi+c_2x_n^2+c_4x_n^4+\cdots
+c_{n}x_n^{n}+\sum_{i=n+1}^\infty
c_{i} x_n^i,
\end{align*}
and, for $n$ odd,
\begin{align*}
u=\varphi+c_2x_n^2+c_4x_n^4+\cdots
+c_{n-1}x_n^{n-1}+\sum_{i=n+1}^\infty
\sum_{j=0}^{N_i}c_{i,j} x_n^i (\log x_n)^j.
\end{align*}
We note that the odd powers $i$ of $x_n$ are absent for $i\le n$ and $c_{n+1, 1}=0$ for $n$ even.
For nonzero $H$, all powers of $x_n$ appear and $c_{n+1,1}$ is present in general except for $n=2$
and constant $H$.

Logarithmic terms also appear in other problems, such as
the singular Yamabe problem
in \cite{ACF1982CMP}, \cite{Loewner&Nirenberg1974} and
\cite{Mazzeo1991}, the complex Monge-Amp\`{e}re equations in \cite{ChengYau1980CPAM},
\cite{Fefferman1976} and \cite{LeeMelrose1982},
and the asymptotically hyperbolic Einstein metrics
in \cite{Anderson2003}, \cite{Biquad2010},
\cite{Chrusciel2005} and
\cite{Hellimell2008}. In fact, Fefferman \cite{Fefferman1976}
observed that logarithmic terms should appear
in the expansion.

Our goal in this paper
is to discuss the relation between $u$ and its formal expansions 
for boundary values
of finite regularity and derive sharp estimates of remainders for the asymptotic expansions.
We will also investigate the regularity property of nonlocal coefficients in the
expansions.

Let $k\ge n+1$ be an integer and set
\begin{align}\label{b1b-v}
u_k=\varphi+c_1x_n+c_2x_n^2+\cdots
+c_{n}x_n^{n}+\sum_{i=n+1}^k
\sum_{j=0}^{\left[\frac{i-1}{n}\right]}c_{i,j} x_n^i (\log x_n)^j,
\end{align}
where $c_i$ and $c_{i,j}$ are 
functions of $x'\in B_1'$.
We point out that the highest order in $u_k$ is given by $x_n^k$.
According to the pattern in this expansion, if we intend to continue to
expand $u_k$, the next term has an order of $x_n^{k+1}(\log x_n)^{\left[\frac{k}{n}\right]}$.
In these expansions,
$c_{n+1,0}$ is the coefficient of the first global term and has no explicit expressions
in terms of $\varphi$.

In this paper, we study the regularity and growth of the remainder $u-u_k$ by following
Han and Jiang
\cite{HanJiang} closely.
We will prove the following result. As mentioned earlier, $H$ in \eqref{eq-Intro-Equation} is
a function in $B_1^+$ instead of a constant.

\begin{theorem}\label{thrm-Main-v}
For some integer $k\ge n+1$ and some constant $\alpha\in (0,1)$,
let $\varphi\in C^{k,\alpha}(B'_1)$ and $H\in C^{k-1,\alpha}(\bar B_1^+)$ be given functions,
with $|H|<1$ in $\bar B_1^+$,
and $u\in C(\bar B^+_1)\cap C^{k,\alpha}(B^+_1)$ be
a solution of \eqref{eq-Intro-Equation}-\eqref{eq-Intro-BoundaryValue}. Then,
there exist functions  $c_i$, $c_{i,j}\in C^{k-i, \epsilon}(B_1')$, for $i=0, 1, \cdots, k$
and any $\epsilon\in (0,\alpha)$,
such that,
for $u_k$ defined as in \eqref{b1b-v},
for any $m=0, 1, \cdots, k$, any $\epsilon\in (0,\alpha)$, and any $r\in (0, 1)$,
\begin{equation}\label{eq-MainRegularity}\partial_{x_n}^m (u-u_k)\in C^{\epsilon}(\bar B^+_r),
\end{equation}
and, for any $(x',x_n)\in
B^+_{1/2}$,
\begin{equation}\label{eq-MainEstimate}|\partial_{x_n}^m (u-u_{k})(x',x_n)|
\le C x_n^{k-m+\alpha},
\end{equation}
for some positive constant $C$ depending only on $n$, $k$, $\alpha$,
the $L^\infty$-norm of $u$ in $B_1^+$, the $C^{k, \alpha}$-norm  of $\varphi$ in
$B_1'$, and the $C^{k-1, \alpha}$-norm  of $H$ in
$\bar B_1^+$. If, in addition, $c_{n+1,1}=0$ on $B_1'$, then
$u\in C^{k,\epsilon}(\bar B^+_r)$ for any $r\in(0,1)$ and any $\epsilon\in (0,\alpha)$.
\end{theorem}

We note that the estimate \eqref{eq-MainEstimate} is optimal and
that there is a slight loss of regularity of $c_{i,j}$ and $u-u_k$, for $i, k\ge n+1$.
In fact,  there is actually no loss of regularity for coefficients of local terms.
If $\varphi\in C^{k,\alpha}(B'_1)$ for some
$k\ge 2$ and $\alpha\in (0,1)$, then
$c_i\in C^{k-i, \alpha}(B_1')$, for $0\le i\le \min\{k,n\}$,
and $c_{n+1,1}\in C^{k-n-1, \alpha}(B_1')$ if $k\ge n+1$.
Moreover,  if $\varphi\in C^{k,\alpha}(B'_1)$ for some
$2\le k\le n$ and $\alpha\in (0,1)$, then
$u\in C^{k,\alpha}(\bar B^+_r)$ for any $r\in(0,1)$.
(See Theorem 3.2 \cite{Tonegawa1996MathZ}.)

As we see, $c_{n+1,1}$, the coefficient of the first logarithmic term in \eqref{b1b-v},
is given by an expression involving derivatives of $\varphi$ up to order
$n+1$. The final part of Theorem \ref{thrm-Main-v} asserts that $c_{n+1,1}\neq 0$ is the obstacle
to the higher regularity of solutions $u$. Theoretically, it is a routine process to calculate
$c_{n+1,1}$. However, it is difficult to identify the geometric meaning of this quantity.
For $n=3$ and $H\equiv0$, Han and Jiang
\cite{HanJiang} demonstrated that $c_{4,1}$ is related to the Willmore functional.
We will prove that $c_{3,1}\equiv 0$ for $n=2$ and constant $H$ in Section \ref{sec-FormalExpansions}.
As a consequence, for $n=2$ and constant $H$, if
$\varphi\in C^{k,\alpha}(B'_1)$ for some integer $k\ge 3$ and some constant $\alpha\in (0,1)$, then
$u\in C^{k,\epsilon}(\bar B^+_r)$ for any $r\in(0,1)$ and any $\epsilon\in (0,\alpha)$.
(See also Theorem 4.1 \cite{Tonegawa1996MathZ}.) We point out that there is a loss
of regularity and it is not clear whether we have   $u\in C^{k,\alpha}(\bar B^+_{1/2})$.

If $\varphi\in C^\infty(B'_1)$, then the estimate \eqref{eq-MainEstimate}
holds for all $m\ge 0$,
all $k\ge \max\{n+1,m\}$ 
and all $\alpha\in (0,1)$. In fact, a similar estimate holds for $x'$-derivatives of
$u-u_k$. This implies in particular that $u$ is
{\it polyhomogeneous}. Refer to
\cite{ACF1982CMP} or \cite{Mazzeo1991} 
for the definition of polyhomogeneity.

We finish the introduction with a brief outline of the paper.
In Section \ref{sec-FormalExpansions},
we provide a calculation to determine all the local terms in the formal expansion.
In Section \ref{Sec-MaximumPrinciple}, we estimate the difference of the
solution and its expansion involving all the local terms. The proof is based on
the maximum principle.
In Section \ref{Sec-TangentialSmooth},
we prove the tangential smoothness
of solutions near boundary by the maximum principle and scaled Schauder estimates.
In Section \ref{sec-N-Regularityies}, we study the regularity along the normal direction
by rewriting the constant mean curvature equation as an ordinary differential equation along the
normal direction.

\section{Formal Expansions}\label{sec-FormalExpansions}

In this section, we derive expansions for solutions of
\eqref{eq-Intro-Equation} and \eqref{eq-Intro-BoundaryValue}. We denote by $x=(x', x_n)$ points in $\mathbb R^n$
and set
\begin{align}\label{FM1}
Q(u)=\Delta u - \frac{u_i u_j}{1+|D u|^2}u_{ij}-\frac{n}{x_n}\big(u_n-H\sqrt{1+|Du|^2}\big).
\end{align}
In the following, we calculate the operator $Q$ on polynomials of $x_n$.
We set
\be\label{eq-Expression_u*}u_*=c_0+c_1x_n+c_2x_n^2+\cdots+c_{n}x_n^{n}
+c_{n+1,1}x_n^{n+1}\log x_n,\ee
where $c_0, c_1,\cdots, c_n$ and $c_{n+1,1}$ are functions of $x'$ to be determined, with
$$c_0=\varphi.$$

\begin{lemma}\label{lemma-FormalExpansion}
Assume $\varphi\in C^\ell(B_1')$ and $H\in C^{\ell-1}(\bar B_1^+)$, with
$|H|<1$ in $\bar B_1^+$, for some $\ell\ge n+3$. Then,
there exist $c_i\in C^{\ell-i}(B_1')$, for $i=1, 2, \cdots, n$, and
$c_{n+1,1}\in C^{\ell-n-1}(B_1')$ such that, for $u_*$
defined in \eqref{eq-Expression_u*},
$$|Q(u_*)|\le Cx_n^{n}\log x_n^{-1},$$
where $C$ is a positive constant depending only on $n$, the
$C^{n+3}$-norm of $\varphi$ in $B_1'$ and the $C^{n+2}$-norm of $H$ in $B_1^+$.
\end{lemma}

\begin{proof}
Set
$$H=H_{0}+H_{1}x_{n}
+\cdots+H_{n}x_{n}^{n}+H_{n+1}x_n^{n+1}+\cdots,$$
where $H_0, H_1,\cdots, H_{n}, H_{n+1}$  are functions of $x'$.
We now substitute $u_*$ given by \eqref{eq-Expression_u*} and $H$ above in $Q$
and arrange $Q(u_*)$ in an ascending order of $x_n$. By requiring the coefficient
of $x_n^{-1}$ to be zero in $Q(u_*)$, we have
\begin{equation}\label{eq-Expression_c_1}
c_{1}=H_{0}\sqrt{\frac{1+|D_{x'}\varphi|^2}{1-H_{0}^{2}}}.\end{equation}
Then,
$$1+|D_{x'}c_0|^2+c_1^2=\frac{1+|D_{x'}\varphi|^2}{1-H_0^2}.$$
For $i=2, \cdots, n$, by requiring the coefficient
of $x_n^{i-2}$ to be zero in $Q(u_*)$ successively, we have
\begin{equation}\label{eq-Expression_c_i}c_{i}=\frac{1}{i(n+1-i)}
F_{i}(c_{0},\cdots,c_{i-1},H_{0},\cdots, H_{i-1}),\end{equation}
where $F_i$ is a smooth function in $\varphi, c_1, \cdots, c_{i-1}$,
$H_{0}, \cdots, H_{i-1}$ and their derivatives up to order 2.
For example,
$$\aligned c_{2}&=\frac{1}{2(n-1)(1-H_{0}^{2})}\bigg[ \Delta_{x'}c_{0}-
\sum_{\alpha,\beta=1}^{n-1}
\frac{\partial_{\alpha}c_{0}\partial_{\beta}c_{0}\partial_{\alpha\beta}c_{0}}
{1+|D_{x'}c_0|^2+c_1^2}\\
&\qquad +\frac{(n-2)H_0D_{x'}c_0\cdot D_{x'}c_1}
{1+|D_{x'}c_0|^2+c_1^2}+nH_{1}\sqrt{1+|D_{x'}c_{0}|^{2}}\bigg].
\endaligned$$
We point out that the coefficient of $x_{n}^{n-1}\log x_{n}$ in $Q(u_*)$ equals 0.
By requiring the coefficient
of $x_n^{n-1}$ to be zero in $Q(u_*)$, we have
\begin{equation}\label{eq-Expression_c_{l+1,1}}
c_{n+1,1}=\frac{1}{n+1}F_{n+1,1}(c_{0},\cdots,c_{n},H_{0},\cdot\cdot\cdot,H_{n}),
\end{equation}
where $F_{n+1,1}$ is a smooth function in
$c_{0},\cdots,c_{n},H_{0},\cdots,H_{n}$ and their derivatives up to order 2.
The next term in $Q(u_*)$ is $x_n^n\log x_n$.  \end{proof}

The functions $c_1, \cdots, c_n$ and $c_{n+1, 1}$ defined in
\eqref{eq-Expression_c_1}, \eqref{eq-Expression_c_i}
and \eqref{eq-Expression_c_{l+1,1}} are functions of $x'\in B_1'$.
We will refer to the corresponding terms as {\it local terms}.
Next, we calculate $c_{3,1}$ for $n=2$ and constant $H$.

\begin{prop}\label{Prop-LogCoefficient}
For $n=2$, if $H$ is constant with $|H|<1$, then $c_{3,1}$ in \eqref{eq-Expression_c_{l+1,1}} is zero.
\end{prop}

\begin{proof} For $n=2$, the operator $Q$ is given by
$$Q(u)=\Delta u - \frac{u_i u_j}{1+|D u|^2}u_{ij}-\frac{2}{x_2}\big(u_2-H\sqrt{1+|Du|^2}\big).$$
We write $x=(x_1, x_2)$ and denote by $w'$ the derivative of $w=w(x_1)$ with respect to $x_1$.
Set
$$u_*=c_0+c_1x_2+c_2x_2^2+c_{3,1}x_2^3\log x_2.$$
We need to calculate $F_{3,1}$ in the proof of Lemma
\ref{lemma-FormalExpansion}. In fact, by a calculation as in the proof of
Lemma \ref{lemma-FormalExpansion}, we have
\begin{align*}
c_{0}&=\varphi,\\
c_{1}&=\frac{H}{\sqrt{1-H^2}}\sqrt{1+\varphi'^{2}},\\
c_{2}&=\frac{(1+\varphi'^{2}H^{2})}{2(1-H^{2})(1+\varphi'^{2})}\varphi'',
\end{align*}
and
\begin{align*}
\begin{split}
c_{3,1}&=-\frac{1}{3(1-H^{2})}
\bigg[c''_{1}\\
&\qquad-\frac{1}{1+c_{1}^{2}+c_{0}'^{2}}
(c''_{1}c_{0}'^{2} +2c''_{0}c'_{1}c'_{0}+4c_{1}c'_{0}c'_{2}
+2c_{1}c_{1}'^2+4c_{2}c'_{0}c'_{1}+8c_{1}c_{2}^{2})
\\
&\qquad+\frac{1}{(1+c_{1}^{2}+c_{0}'^{2})^2}(4c_{2}c_{1}+2c'_{1}c'_{0})
(c''_{0}c_{0}'^{2}+2c_{1}c'_{0}c'_{1}
+2c_{1}^{2}c_{2})\\
&\qquad+H\bigg(
\frac{4c_{2}^{2}+2c'_{0}c'_{2}+c_{1}'^{2}}{(1+c_{1}^{2}+c_{0}'^{2})^{\frac{1}{2}}}
-\frac{(2c_{1}c_{2}+c'_{0}c'_{1})^{2}}{(1+c_{1}^{2}+c_{0}'^{2})^{\frac{3}{2}}}\bigg)\bigg].
\end{split}
\end{align*}
By the expression of $c_0, c_1$ and $c_2$, we have
\begin{align*}
c'_{1}&=\frac{H}{\sqrt{1-H^{2}}}\frac{\varphi'\varphi''}{\sqrt{1+\varphi'^{2}}},\\
c''_{1}&=\frac{H}{\sqrt{1-H^{2}}}\bigg[\frac{\varphi''^{2}}{(\sqrt{1+\varphi'^{2}})^{3}}
+\frac{\varphi'\varphi'''}{\sqrt{1+\varphi'^{2}}}\bigg],\\
c'_{2}&=\frac{1+\varphi'^{2}H^{2}}{2(1-H^{2})(1+\varphi'^{2})}\varphi'''-\frac{\varphi'\varphi''^2}{(1+\varphi'^{2})^{2}}.
\end{align*}
We substitute these expression in $c_{3,1}$ and collect terms involving $\varphi'''$ first and then $\varphi''$.
We obtain $c_{3,1}=0$ after a lengthy calculation. \end{proof}


\section{Estimates of Local Terms}\label{Sec-MaximumPrinciple}

In this section, we derive an estimate for an expansion involving all local
terms by the
maximum principle.
For convenience, we write  $t=x_n$  and denote by  $x=(x',t)$ points in $\mathbb R^n$. Set
$Q$ as in \eqref{FM1}, i.e.,
$$Q(u)=\Delta u - \frac{u_i u_j}{1+|D u|^2}u_{ij}-\frac{n}{t}\big(u_n-H\sqrt{1+|Du|^2}\big).$$
Throughout the paper, we always assume
$$|H|<1\quad\text{in }\bar B_1^+.$$

First, we derive a decay estimate by  the maximum principle.

\begin{lemma}\label{lemma-order one}
Assume $\varphi\in C^1(B_1')$ and $H\in C(\bar B_1^+)$ with $|H|<1$ in $\bar B_1^+$.  Let
$ u \in C(\bar{B}_{1}^{+})\bigcap C^{2}(B_{1}^{+})$ be a solution of
\eqref{eq-Intro-Equation}-\eqref{eq-Intro-BoundaryValue}.
Then, for any $(x',t)\in B_{1/4}'\times(0,1/4),$
\begin{align}\label{real1}|u-\varphi|\le Ct,\end{align}
where $C$ is a positive constant
depending only on n, $|\varphi|_{C^{1}(B_{{3}/{4}}')},$
$|u|_{L^{\infty}(B_{1}^{+})}$ and
$|H|_{L^{\infty}(B_{1}^{+})}$.
\end{lemma}

\begin{proof}
Take any $x_{0}'\in B_{1/4}'$ and set $r_0= 1/4$ and $\delta_{0}= 1/4.$ Then,
${B_{r_0}'(x_{0}')}\times(0,\delta_{0})\subset B_{3/4}^+.$
Set
$$\overline{w}(x',t)=\varphi(x_0')+At+B(|x'-x_{0}'|^{2}+t^2)^{\frac{1}{2}},$$
where $A$ and $B$ are positive constants to be determined.

We now consider $u$ on  $\partial( B_{r_0}'(x_{0}')\times(0,\delta_{0}))$.  First,
$$u(x',0)=\varphi(x')\leq\varphi(x_0')+B|x'-x_0'|=\overline{w}(x',0)\quad\text{on }B_{r_0}'(x_{0}'),$$
if $B\geq |\varphi|_{C^{1}(B_{{1}/{2}}')}.$ Next
$$u(x',t)\leq\varphi(x_0')+Br_0\leq\overline{w}(x',t)\quad\text{on }\partial B_{r_0}'(x_{0}')\times(0,\delta_{0}),$$
if $$B\geq r_0^{-1}\big(|u|_{L^{\infty}(B_{1}^{+})} + |\varphi|_{L^{\infty}(B_{{3}/{4}}')}\big).$$
Last,
$$u(x',\delta_{0})\leq\varphi(x_0')+A\delta_0\leq\overline{w}(x',\delta_{0}),$$
if $$A\geq \delta_0^{-1}\big(|u|_{L^{\infty}(B_{1}^{+})} + |\varphi|_{L^{\infty}(B_{{3}/{4}}')}\big).$$
Therefore, we set
\begin{align*}
B&=\max\{ |\varphi|_{C^{1}(B_{{1}/{2}}')},
r_0^{-1}\big(|u|_{L^{\infty}(B_{1}^{+})} + |\varphi|_{L^{\infty}(B_{{3}/{4}}')}\big)\},\\
A&=\widetilde{A}\max\{B, \delta_0^{-1}\big(|u|_{L^{\infty}(B_{1}^{+})} + |\varphi|_{L^{\infty}(B_{{3}/{4}}')}\big)\},
\end{align*}
where $\widetilde{A}\geq 1$ is a  constant to be determined.
Therefore, $$u\leq\overline{w}\quad\text{on }\partial (B_{r_0}'(x_{0}')\times(0,\delta_{0})).$$

Next, we calculate $tQ(\overline{w})$.
A straightforward calculation yields
\begin{align*}
\overline{ w}_n&=A+\frac{Bt}{(|x'-x_0'|^{2}+t^2)^{\frac{1}{2}}},\\
\overline{ w}_a&=\frac{B(x_a-x_{0a})}{(|x'-x_0'|^{2}+t^2)^{\frac{1}{2}}},\end{align*}
and
\begin{align*}
t\overline{ w}_{nn}&=\frac{Bt}{(|x'-x_0'|^{2}+t^2)^{\frac{1}{2}}}-\frac{Bt^3}{(|x'-x_0'|^{2}+t^2)^{\frac{3}{2}}},\\
t\overline{ w}_{na}&=-\frac{Bt^2(x_a-x_{0a})}{(|x'-x_0'|^{2}+t^2)^{\frac{3}{2}}},\\
t\overline{ w}_{ab}&=\frac{B\delta_{ab}t}{(|x'-x_0'|^{2}+t^2)^{\frac{1}{2}}}
-\frac{Bt(x_a-x_{0a})(x_b-x_{0b})}{(|x'-x_0'|^{2}+t^2)^{\frac{3}{2}}}.
\end{align*}
Write
$$-n\big(\overline{w}_n-H\sqrt{1+|D\overline{w}|^2}\big)=
n\big(\sqrt{1+|D\overline{w}|^2}-\overline{w}_n\big)-n(1-H)\sqrt{1+|D\overline{w}|^2}.$$
Then,
$$tQ(\overline{w})\le t|D^2\overline{w}|
+n\big(\sqrt{1+|D\overline{w}|^2}-\overline{w}_n\big)-n(1-|H|)\sqrt{1+|D\overline{w}|^2}.$$
Note
$$1+|D\overline{w}|^2\ge \overline{w}_n^2\ge A^2,$$
and
$$\sqrt{1+|D\overline{w}|^2}-\overline{w}_n\le \sqrt{1+|D_{x'}\overline{w}|^2}\le \sqrt{1+B^2}.$$
Then,
\begin{align*}
tQ(\overline{w})&\leq n^2B+n\sqrt{1+B^2}-nA(1-|H|).\end{align*}
Since $\sup|H|<1,$ by taking $A$ large or $\widetilde{A}$ large, we have
$$tQ(\overline{w})\leq 0\quad\text{in }B_{r_{0}}'(x_{0}')\times(0,\delta_{0}).$$

By the maximum principle,  we obtain
$$u(x',t)\leq\overline{w}(x',t) \quad\text{in }B_{r_{0}}'(x_{0}')\times(0,\delta_{0}).$$
Evaluating at $x'=x'_{0},$ we get
$$u-\varphi\leq At+Bt \quad \text{in } B_{1/4}'\times(0,\delta_{0}).$$
For the lower bound of $u,$ we set
$$\underline{w}(x',x_{n})=\varphi(x_0')-At-B(|x'-x_0'|^{2}+t^2)^{\frac{1}{2}},$$
and proceed similarly.
\end{proof}

Next, we improve Lemma \ref{lemma-order one} under better assumptions of $\varphi$ and $H$.

\begin{lemma}\label{lemma-order 2}
Assume $\varphi\in C^{2}(B_{1}')$ and $H\in C^1(\bar B_1^+)$, with
$|H|<1$ in $\bar B_1^+$. Let
$ u \in C(\bar{B}_{1}^{+})\bigcap C^{2}(B_{1}^{+})$ be a solution of
\eqref{eq-Intro-Equation}-\eqref{eq-Intro-BoundaryValue}.
Then, for any $(x',t)\in B_{1/8}'\times(0,\delta_{0}),$
\begin{align}\label{real2}|u-\varphi-c_{1}(x')t|\leq Ct^{1+\varepsilon},\end{align}
where $c_1$ is given by \eqref{eq-Expression_c_1},
and $\delta_{0},$ $\varepsilon$ and $C$ are positive constants
depending only on n, $|\varphi|_{C^{2}(B_{{3}/{4}}')},$
$|u|_{L^{\infty}(B_{1}^{+})}$ and
$|H|_{C^{1}(B_{1}^{+})}$.
\end{lemma}

\begin{proof}
Take any $x_{0}'\in B_{1/8}'$ and set
$r_0= 1/8$ and $\delta_{0}$ small to be determined. Consider in  $B_{r_0}'(x_{0}')\times(0,\delta_{0}).$
Set
$$\overline{w}(x',t)=\varphi(x')+c_{1}(x_0')t+At^{1+\varepsilon}+B(|x'-x_{0}'|^{2}+t^2)^{\frac{\varepsilon}{2}}t,$$
where $A$ and $B$ are positive constants to be determined.
Obviously, $u(x',0)=\overline{w}(x',0).$
In order to have $u\leq\overline{w}$ on $\partial( B_{r_0}'(x_{0})\times(0,\delta_{0})),$
we require,  by Lemma \ref{lemma-order one},
\begin{align}\label{1}
C_{0}+|c_{1}|_{L^{\infty}(B_{1/2}')}\leq Br_0^{\varepsilon},
\end{align}
and
\begin{align}\label{2}
C_{0}+|c_{1}|_{L^{\infty}(B_{1/2}')}\leq A\delta_{0}^{\varepsilon},
\end{align}
where $C_{0}$ is the constant in Lemma \ref{lemma-order one}.

A straightforward calculation yields, for $a,b\neq n,$
\begin{align*}
\overline{w}_{a}&=\partial_{a}\varphi+B\varepsilon\frac{(x_{a}-x_{0a})t}{(|x'-x_{0}'|^{2}+t^2)^{1-\frac{\varepsilon}{2}}},\\
\overline{w}_{n}&=c_{1}(x_0')+(1+\varepsilon)At^{\varepsilon}
+B(|x'-x_{0}'|^{2}+t^2)^{\frac{\varepsilon}{2}}+B\varepsilon\frac{t^2}{(|x'-x_{0}'|^{2}+t^2)^{1-\frac{\varepsilon}{2}}},
\end{align*}
and
\begin{align*}
t\overline{w}_{ab}&=\partial_{ab}\varphi t+B\varepsilon\frac{\delta_{ab}t^2}{(|x'-x_{0}'|^{2}+t^2)^{1-\frac{\varepsilon}{2}}}+
B\varepsilon(-2+\varepsilon)\frac{(x_{b}-x_{0b})(x_{a}-x_{0a})t^2}{(|x'-x_{0}'|^{2}+t^2)^{2-\frac{\varepsilon}{2}}},\\
t\overline{w}_{na}&=B\varepsilon\frac{(x_{a}-x_{0a})t}{(|x'-x_{0}'|^{2}+t^2)^{1-\frac{\varepsilon}{2}}}
+B\varepsilon(-2+\varepsilon)\frac{(x_{a}-x_{0a})t^3}{(|x'-x_{0}'|^{2}+t^2)^{2-\frac{\varepsilon}{2}}},\\
t\overline{w}_{nn}&
=\varepsilon(1+\varepsilon)At^{\varepsilon}+B\varepsilon\frac{3t^2}{(|x'-x_{0}'|^{2}+t^2)^{1-\frac{\varepsilon}{2}}}
+B\varepsilon(-2+\varepsilon)\frac{t^4}{(|x'-x_{0}'|^{2}+t^2)^{2-\frac{\varepsilon}{2}}}.
\end{align*}
In analyzing $Q(\overline{w})$, we first consider the terms involving the second derivatives of $\overline{w}$.
We note that $\overline{w}_a$, $\overline{w}_{ab}$ and
$\overline{w}_{an}$ do not involve $A$, and
$$\frac{|\partial_i\overline{w}\partial_j\overline{w}|}{1+|D\overline{w}|^2}\le \frac12.$$
By considering $\partial_{nn}\overline{w}$ separately, we have
\begin{align}\label{eq-Estimate01}
t\bigg[\Delta{\overline{w}}-\frac{\overline{w}_{i}\overline{w}_{j}\overline{w}_{ij}}{1+| D\overline{w}|^{2}}\bigg]
\leq Ct+CB\varepsilon t^{\varepsilon}
+\varepsilon(1+\varepsilon)At^{\varepsilon},
\end{align}
where $C$ is a positive constant depending only on $n$, 
$|\varphi|_{C^{2}}$. 

Next, we claim, by choosing $A$ and $\delta_{0}$ appropriately,
\begin{align}\label{eq-Estimate02}
\begin{split}
&-\overline{w}_{n}
+H\sqrt{1+|D\overline{w}|^{2}}\\
&\qquad\leq -(1-|H|)\bigg[(1+\varepsilon) At^{\varepsilon}+B(|x'-x_{0}'|^{2}+t^2)^{\frac{\varepsilon}{2}}
+B\varepsilon\frac{t^2}{(|x'-x_{0}'|^{2}+t^2)^{1-\frac{\varepsilon}{2}}}\bigg]
\\&\qquad\qquad+Ct+CB\varepsilon t^{\varepsilon}+C|x'-x_0'|,
\end{split}
\end{align}
where $C$ is a positive constant depending only on $n$, $|u|_{L^{\infty}}$,
$|\varphi|_{C^{2}}$ and $|H|_{C^{1}}.$
To prove this, we consider two cases $H>0$ and $H<0$ separately.
Set
$$D=(1+\varepsilon) At^{\varepsilon}+B(|x'-x_{0}'|^{2}+t^2)^{\frac{\varepsilon}{2}}
+B\varepsilon\frac{t^2}{(|x'-x_{0}'|^{2}+t^2)^{1-\frac{\varepsilon}{2}}}.$$
Then,
$$\overline{w}_{n}=c_{1}(x_0')+D,$$
and we need to prove
\begin{equation}\label{eq-Estimate1}H\sqrt{1+|D\overline{w}|^{2}}\le c_1(x_0')+|H|D
+Ct+CB\varepsilon t^{\varepsilon}+C|x'-x_0'|.\end{equation}

{\it Case 1.} We assume $H(x',t)\geq0.$ Note
\begin{align*}
1+|D\overline{w}|^{2}&\leq 1+
|D_{x'}\varphi|^{2}+c_{1}^2(x_0')+D^{2}+2Dc_{1}(x_0')+CB\varepsilon t^{\varepsilon}\\
&\leq1+|D_{x'}\varphi|^{2}+c_{1}^{2}(x')+C|x'-x_0'|
+2\sqrt{\frac{1+|D_{x'}\varphi|^{2}}{1-H_0^{2}}}(x_0')D+D^{2}+CB\varepsilon t^{\varepsilon}.
\end{align*}
By  \eqref{eq-Expression_c_1}, we have
\begin{align*}
1+|D_{x'}\varphi|^{2}+c_{1}^{2}=\frac{1+|D_{x'}\varphi|^{2}}{1-H_0^{2}},
\end{align*}
and
$$|c_1(x_0')|=|H_0(x_0')|\sqrt{\frac{1+|D_{x'}\varphi|^{2}}{1-H_0^{2}}}(x_0')<\sqrt{\frac{1+|D_{x'}\varphi|^{2}}{1-H_0^{2}}}(x_0').$$
Then,
\begin{align*}
1+|D\overline{w}|^{2}\le \left(\sqrt{\frac{1+|D_{x'}\varphi|^{2}}{1-H_0^{2}}}(x_0')+D+C|x'-x_0'|+CB\varepsilon t^{\varepsilon}\right)^2,
\end{align*}
and hence
$$\sqrt{1+|D\overline{w}|^{2}}\le \sqrt{\frac{1+|D_{x'}\varphi|^{2}}{1-H_0^{2}}}(x_0')+D+C|x'-x_0'|+CB\varepsilon t^{\varepsilon}.$$
We now multiply by $H$ and write $H$ in the first term in the right-hand side as
$$H=H_0(x_0')+(H-H_0(x_0')).$$ We
have \eqref{eq-Estimate1}  by adjusting $C$.

{\it Case 2.}
Assume $H(x',t)<0.$
Set $H=-\widetilde{H}.$
We need to prove
$$-c_{1}(x_0')-\widetilde HD\leq \widetilde{H}\sqrt{1+|D\overline{w}|^{2}}+C|x'-x_0'|+CB\varepsilon t^{\varepsilon}+Ct.$$
Similarly as in Case 1, we have
\begin{align*}
1+|D\overline{w}|^{2}&\ge1+|D_{x'}\varphi|^{2}+c_{1}^{2}(x')-C|x'-x_0'|+2c_1(x_0')D+D^{2}-CB\varepsilon t^{\varepsilon}\\
&\ge
\frac{1+|D_{x'}\varphi|^{2}}{1-H_0^{2}}(x_0')
-2|H_0(x_0')|\sqrt{\frac{1+|D_{x'}\varphi|^{2}}{1-H_0^{2}}}(x_0')D\\
&\qquad+D^{2}-C|x'-x_0'|-CB\varepsilon t^{\varepsilon}.
\end{align*}
Write
$$-|H_0(x_0')|=-1+(1-|H_0(x_0')|),$$
and note
$$(1-|H_0(x_0')|)\sqrt{\frac{1+|D_{x'}\varphi|^{2}}{1-H_0^{2}}}(x_0')\ge (1-|H|_{L^\infty}).$$
Hence,
\begin{align*}1+|D\overline{w}|^{2}\ge \left(\sqrt{\frac{1+|D_{x'}\varphi|^{2}}{1-H_0^{2}}}(x_0')-D\right)^2
+(1-|H|_{L^\infty})D-C|x'-x_0'|-CB\varepsilon t^{\varepsilon}.
\end{align*}
By
$$D\ge B(|x'-x_{0}'|^{2}+t^2)^{\frac{\varepsilon}{2}},$$
and choosing $\varepsilon$ small and $B$ large such that
\begin{align}\label{4,5}
  \begin{split}
    C\varepsilon&<\frac{1}{2}(1-|H|_{L^{\infty}}),\\
    C&<\frac{1}{2}B(1-|H|_{L^{\infty}}),
  \end{split}
\end{align}
we have
$$\sqrt{1+|D\overline{w}|^{2}}\ge \sqrt{\frac{1+|D_{x'}\varphi|^{2}}{1-H_0^{2}}}(x_0')-D.$$
We now multiply by $\widetilde H=\widetilde H+H(x_{0}',0)-H(x_{0}',0).$ Then
$$-c_{1}(x_0')-\widetilde HD-Ct-C|x'-x_0'|\leq \widetilde{H}\sqrt{1+|D\overline{w}|^{2}}.$$
We have \eqref{eq-Estimate1} by \eqref{eq-Expression_c_1} and by adjusting $C$.

By combining \eqref{eq-Estimate01} and \eqref{eq-Estimate02}, we obtain
\begin{align*}
tQ(\overline{w})&\leq Ct+CB\varepsilon t^{\varepsilon}+C|x'-x_0'|
-n(1-|H|)B(|x'-x_{0}'|^{2}+t^2)^{\frac{\varepsilon}{2}}\\
&\qquad+(1+\varepsilon) At^{\varepsilon}(\varepsilon-n(1-|H|)).
\end{align*}
Requiring $\delta_0,\varepsilon$ small and $B$ large such that
\begin{align}\label{3,6,7,8}
   \begin{split} \delta_0 &\ll1,\\
    C\varepsilon&<\frac14n(1-|H|_{L^\infty}),\\
    C&<\frac14Bn(1-|H|_{L^\infty}),\\
    \varepsilon&<\frac12n(1-|H|_{L^\infty}),\end{split}
\end{align}
then we have
$tQ(\overline{w})\leq 0.$

In conclusion, we choose $\varepsilon,B,\delta_0,A,$ successively
such that they satisfy \eqref{1}, \eqref{2}, \eqref{4,5} and \eqref{3,6,7,8}.
We now apply the maximum principle to conclude
$$u\leq\overline{w}\quad\text{in }B_{r_0}'(x_{0}')\times(0,\delta_{0}).$$
For the lower bound, we set
$$\underline{w}(x',t)=\varphi(x')+c_{1}(x_0')t-At^{1+\varepsilon}-B(|x'-x_{0}'|^{2}+t^2)^{\frac{\varepsilon}{2}}t,$$
and proceed similarly. Then,
$$u\geq\underline{w}\quad\text{in }B_{r_0}'(x_{0}')\times(0,\delta_{0}).$$
Hence, \begin{align*}|u-\varphi-c_{1}(x_0')t|\leq At^{1+\varepsilon}+B(|x'-x_{0}'|^{2}+t^2)^{\frac{\varepsilon}{2}}t.\end{align*}
By taking $x'=x'_{0},$ we obtain the desired result.
\end{proof}

We now improve Lemma \ref{lemma-order 2} under the same assumption.

\begin{lemma}\label{t2}
Assume $\varphi\in C^{2}(B_{1}')$ and $H\in C^1(\bar B_1^+)$, with
$|H|<1$ in $\bar B_1^+$. Let
$ u \in C(\bar{B}_{1}^{+})\bigcap C^{2}(B_{1}^{+})$ be a solution of
\eqref{eq-Intro-Equation}-\eqref{eq-Intro-BoundaryValue}.
Then, for any $(x',t)\in B_{r_1}'\times(0,\delta_1),$
\begin{align*}|u-\varphi-c_{1}t|\leq Ct^{2},\end{align*}
where $c_1$ is given by \eqref{eq-Expression_c_1},
and $r_1,\delta_1$ and $C$ are positive constants
depending only on n, $|\varphi|_{C^{2}(B_{{3}/{4}}')},$
$|u|_{L^{\infty}(B_{1}^{+})}$ and
$|H|_{C^{1}(B_{1}^{+})}$.
\end{lemma}

\begin{proof}
Take any $x'_{0}\in B_{r_1}'$ and let $r_1,\delta_1$ be small positive constants to be determined. Set
$$D= B_{r_1}'(x_{0}')\times(0,\delta_1).$$ Take $\varepsilon$ as in Lemma 3.2. Set
\begin{align}\label{a1}
r_1=\delta_1^{\frac{1+\varepsilon}2},
\end{align}
and take $\delta_1$ small so that
$r_1<{1}/{16}$.
Set
\begin{align*}
v(x',t)&=\varphi(x')+c_{1}(x_0')t,\\
\phi(x',t)&=At^2+Bt(t^{2}+|x'-x'_{0}|^{2})^\frac{1}{2},
\end{align*}
and
$$\overline{w}=v+\phi.$$
We will prove $u\leq \overline{w}$ in $D$.

Consider on $\partial D.$ Obviously, $\overline{w}(x',0)=\varphi(x')=u(x',0).$
Since $c_1$ is $C^1$ in $x'$, we have
$$\overline{w}(x',t)\ge \varphi(x')+c_{1}(x')t-C|x'-x_0'|t+At^2+B|x'-x'_{0}|t.$$
By requiring $B$ large such that
\begin{align}\label{a2}
    B>2C,
\end{align}
we have
$$\overline{w}(x',t)\ge \varphi(x')+c_{1}(x')t+At^2+\frac{B}{2}|x'-x'_{0}|t.$$
Therefore,
\begin{align}\label{e1}
\overline{w}(x',\delta_1)\ge \varphi(x')+c_{1}(x')\delta_1+A\delta_1^2\ge u(x',\delta_1),
\end{align}
 if $A\delta_1^2\ge C_0\delta_1^{1+\varepsilon},$ where $C_0$ is as in (\ref{real2}).
Set $\widetilde{  A}$ to be determined and require
\begin{align}\label{a3}
\widetilde{  A}>C_0.
\end{align}
Take $A=\widetilde{  A}\delta_1^{-(1-\varepsilon)}$. Then,  \eqref{e1} holds.
On $\partial B_{r_1}'(x_{0}')\times(0,\delta_1),$
\begin{align}\label{e2}
\overline{w}(x',t)\ge \varphi(x')+c_{1}(x')t+\frac{1}{2}Br_1t\ge u(x',t),
\end{align}
if $Br_1t/2\ge C_0t^{1+\varepsilon},$ by \eqref{real2}. Using (\ref{a1}), we  require
\begin{align}\label{a4}
   B\ge 2C_0\delta_1^{-\frac{1-\varepsilon}2},
\end{align}
so that \eqref{e2} holds. Set $\widetilde{  B}$ to be determined and require
\begin{align}\label{a5}
\widetilde{  B}>2C_0+2C.
\end{align}
Take $B=\widetilde{  B}\delta_1^{-\frac{1-\varepsilon}{2}}$.
Then, \eqref{a2} and \eqref{a4} are satisfied. In fact, by taking $$\widetilde{A}=\widetilde{B},$$
we have \eqref{a2}, \eqref{a3} and \eqref{a4}.
Now $B=\widetilde{  A}\delta_1^{-\frac{1-\varepsilon}{2}}$ and $\widetilde{A}$ is a large positive constant to be determined.
Hence,
\begin{align}\label{a6}
 \frac{B}{A}=\delta_1^{\frac{1-\varepsilon}{2}}
\end{align}
is small when $\delta_1$ is small.

Next, we calculate $tQ(\overline{w})$. We write
\begin{equation}\label{eq-Expression-Q0}tQ(\overline{w})=tQ(v)+D_{1}+D_{2},\end{equation}
where
$$
D_1=t\bigg(\Delta\phi-\frac{(v_i+\phi_i)(v_j+\phi_j)}{1+|D(v+\phi)^2|}\phi_{ij}
-\bigg[\frac{(v_i+\phi_i)(v_j+\phi_j)}{1+|D(v+\phi)|^2}-\frac{v_iv_j}{1+|Dv|^2}\bigg]v_{ij}\bigg),$$
and
$$D_2=-n\partial_n\phi+nH\big(\sqrt{1+|D(v+\phi)|^{2}}-\sqrt{1+|Dv|^{2}}\big).$$
By the choice of $c_1$, a simple calculation yields
\begin{align}\label{e3}
    tQ(v)\le C(t+|x'-x_0'|).
\end{align}
A straightforward calculation yields,
for $a, b\neq n$.
\begin{align*}
\partial_{a}\phi&=\frac{Bt(x_{a}-x_{0a})}{(t^{2}+|x'-x'_{0}|^{2})^\frac{1}{2}}, \\
\partial_{n}\phi&=2At+B(t^{2}+|x'-x'_{0}|^{2})^\frac{1}{2}+\frac{Bt^2}{(t^{2}+|x'-x'_{0}|^{2})^\frac{1}{2}},
\end{align*}
and
\begin{align*}
t\partial_{ab}\phi&=\frac{B\delta_{ab}t^2}{(t^{2}+|x'-x'_{0}|^{2})^\frac{1}{2}}
-\frac{Bt^2(x_{a}-x_{0a})(x_{b}-x_{0b})}{(t^{2}+|x'-x'_{0}|^{2})^\frac{3}{2}},\\
t\partial_{an}\phi&=\frac{Bt(x_{a}-x_{0a})}{(t^{2}+|x'-x'_{0}|^{2})^\frac{1}{2}}
-\frac{Bt^3(x_{a}-x_{0a})}{(t^{2}+|x'-x'_{0}|^{2})^\frac{3}{2}},\\
t\partial_{nn}\phi&=2At+\frac{3Bt^2}{(t^{2}+|x'-x'_{0}|^{2})^\frac{1}{2}}-\frac{Bt^4}{(t^{2}+|x'-x'_{0}|^{2})^\frac{3}{2}}.
\end{align*}
Hence, we have
\begin{align}\label{b1}
  |\partial_{a}\phi|\le \widetilde{A}\delta_1^{\frac{1}{2}+\frac{\varepsilon}{2}},
\end{align}
and
\begin{align}\label{b2}
 |\partial_{n}\phi|\le  4\widetilde{A}\delta_1^{\frac{1}{2}+\frac{\varepsilon}{2}}+2\widetilde{A}\delta_1^\varepsilon.
\end{align}
Let $$\widetilde{A}=\delta_1^{-\frac{\varepsilon}{2}}.$$
In the following, we always assume $\delta_1$ is small. Then, \eqref{a3} and \eqref{a5} are satisfied and
\begin{align}\label{b3}
 |D\phi|\le  C\delta_1^\frac{\varepsilon}{2}\ll1.
\end{align}
By the Taylor expansion, we have
\begin{align*}
\sqrt{1+|D(v+\phi)|^{2}}-\sqrt{1+|Dv|^2}=\frac{Dv\cdot D\phi+|D\phi|^2/2}{\sqrt{1+|Dv|^2}}
+O\bigg(\frac{(Dv\cdot D\phi+|D\phi|^2)^2}{(1+|Dv|^2)^{3/2}}\bigg).
\end{align*}
Note
\begin{align*}
 1+|Dv|^{2}=1+|D_{x'}\varphi|^2+c_1^2(x'_{0})=\frac{1+|D_{x'}\varphi|^2}{1-H_0^2}(x'_{0})
 +O(|x'-x'_{0}|).
\end{align*}
Hence,
$$\frac{\partial_nv\partial_n\phi}{\sqrt{1+|Dv|^2}}=\partial_n\phi H_0(x'_{0})+O(|\partial_n\phi||x'-x'_{0}|).$$
Therefore,
\begin{align}\label{c1}
\begin{split}
&\sqrt{1+|D(v+\phi)|^{2}}-\sqrt{1+|Dv|^{2}}\\
&\qquad=\partial_n\phi H_0(x'_{0})+
O(|D_{x'}\phi|)+O(|\partial_n\phi||x'-x'_{0})+O(|D\phi|^2).
\end{split}
\end{align}
We also have
\begin{align}\label{c2}\begin{split}
 1-\frac{(v_n+\phi_n)^2}{1+|D(v+\phi)|^2}&=\frac{1+|D_{x'}\varphi|^2(x')}{1+|D_{x'}\varphi|^2(x')+c_1^2(x'_{0})}+O(|D\phi|)\\
&=1-H_0^2(x')+O(|D\phi|+|x'-x_0'|),\end{split}
\end{align}
Note $v_{ij}\neq0$ only if $i,j\neq n .$ 
By \eqref{c2} and $|D_{ai}\phi|\le 2B$, we have
$$D_1\le 2At\big(1-H_0^2(x'_{0})+C|D\phi|+C|x'-x_0'|\big)+CBt.$$
By \eqref{c1}, we get
$$\aligned D_2&\le-n\partial_n\phi\big(1-HH_0(x'_{0})-(C|x'-x_0'|+C|D\phi|)\big)
\\&\le -n(2At+B(t^{2}+|x'-x'_{0}|^{2})^\frac{1}{2}\\
&\qquad+\frac{Bt^2}{(t^{2}+|x'-x'_{0}|^{2})^\frac{1}{2}}
\big(1-H_0^2(x'_{0})-(C|x'-x_0'|+Ct+C|D\phi|)\big)\\
&\le2At\big(-n(1-H_0^2(x'_{0}))+C|x'-x_0'|+Ct+C|D\phi|\big)\\
&\qquad+B(t^{2}+|x'-x'_{0}|^{2})^\frac{1}{2}\big(-n(1-H_0^2(x'_{0}))+C|x'-x_0'|+Ct+C|D\phi|\big).\endaligned$$
Then, we have
\begin{align*}
D_1+D_2&\le 2At\{(1-H_0^2(x'_{0}))(1-n)+C|D\phi|+C|x'-x_0'|+Ct\}\\
&\qquad+B|x'-x'_{0}|\big(-\frac{n}{2}(1-H_0^2(x'_{0}))\big)+CBt.\end{align*}
Choose $\delta_1$ small. Then by definitions of $A,B$ and (\ref{a1}), (\ref{b3}), (\ref{e3}) and (\ref{a6}), we have
$$\aligned tQ(\overline{w})&\le 2At\big((1-H_0^2(x'_{0}))(1-n)+C\delta_1^\frac{\varepsilon}{2}\big)
+B|x'-x'_{0}|\big(-\frac{n}{2}(1-H_0^2(x'_{0}))+C\delta_1^\frac{1-\varepsilon}{2}\big)\\
&\le 2At\big((1-|H|_{L^{\infty}})(1-n)+C\delta_1^\frac{\varepsilon}{2}\big)
+B|x'-x'_{0}|\big(-\frac{n}{2}(1-|H|_{L^{\infty}}^2)+C\delta_1^\frac{1-\varepsilon}{2}\big).\endaligned$$
Then we have, for small $\delta$,
$$Q(\overline{w})=Q(v+\phi)\le 0\quad\text{in }D.$$
By the maximum principle, we get
\begin{align*}\begin{split}u\leq v+\phi\leq\varphi(x')+c_{1}(x_0')t+At^2+B(t^{2}+|x'-x'_{0}|^{2})^\frac{1}{2}t.\end{split}\end{align*}
For the lower bound, we consider $v-\phi$ and, by proceeding similarly, we get
\begin{align*}\begin{split}u&\geq \varphi(x')+c_{1}(x_0')t-At^2-B(t^{2}+|x'-x'_{0}|^{2})^\frac{1}{2}t.\end{split}\end{align*}
By taking $x'=x'_{0},$ we have the desired result.
\end{proof}

Next, we prove an estimate for an expansion of solutions involving all the local terms
by the maximum principle. The proof is similar to that of Lemma \ref{t2}.

\begin{theorem}\label{thrm-order n+1}
Assume $\varphi\in C^{n+3}(B_{1}')$ and $H\in C^{n+2}(\bar B_1^+)$, with
$|H|<1$ in $\bar B_1^+$. Let
$ u \in C(\bar{B}_{1}^{+})\bigcap C^{n+3}(B_{1}^{+})$ be a solution of
\eqref{eq-Intro-Equation}-\eqref{eq-Intro-BoundaryValue}.
Then, for any $(x',t)\in B_{1/16}'\times(0,\delta),$
\begin{align}\label{n+1}|u-\varphi-c_{1}t-c_{2}t^{2}-\cdots-c_{n}t^{n}
-c_{n+1,1}t^{n+1}\log t|\leq Ct^{n+1},\end{align}
where $c_1, \cdots, c_n$ and $c_{n+1,1}$ are functions in $B_1'$
as given in Lemma \ref{lemma-FormalExpansion},
and $\delta$ and $C$ are positive constants
depending only on n, $|\varphi|_{C^{n+3}(B_{1}')},$
$|u|_{L^{\infty}(B_{1}^{+})}$ and
$|H|_{C^{n+2}(B_{1}^{+})}$.
\end{theorem}

\begin{proof}
Take any $x'_{0}\in B_{1/16}'$ and $\delta$ a small positive constant to be determined.  Consider in
$D= B_{\sqrt{\delta}}'(x_{0}')\times(0,\delta).$
Set
$$v(x',t)=c_{0}(x')+c_{1}(x')t+c_{2}(x')t^{2}+\cdot\cdot\cdot+c_{n}(x')t^{n}
+c_{n+1,1}(x')t^{n+1}\log t,$$
and
$$\phi(x',t)=A(|x'-x'_{0}|^{2}+t)^{n+1}-A(|x'-x'_{0}|^{2}+t)^{q}.$$
We will prove $u\leq v+\phi.$

For convenience, we write
$$f=|x'-x'_{0}|^{2}+t.$$ Hence,
$\phi=Af^{n+1}-Af^{q}$ and $f\leq2\delta$ in $D$.
A straightforward calculation yields,
for $a, b\neq n$,
\begin{align*}
\partial_{a}\phi&=2(n+1)Af^{n}(x'_{a}-x'_{0a})-2qAf^{q-1}(x'_{a}-x'_{0a}), \\
\partial_{n}\phi&=(n+1)Af^{n}-qAf^{q-1},
\end{align*}
and
\begin{align*}
\partial_{ab}\phi&=4n(n+1)Af^{n-1}(x'_{a}-x'_{0a})(x'_{b}-x'_{0b})+2(n+1)Af^{n}\delta_{ab}
\\&\qquad-4q(q-1)Af^{q-2}(x'_{a}-x'_{0a})(x'_{b}-x'_{0b})-2qAf^{q-1}\delta_{ab},\\
\partial_{an}\phi&=2(n+1)nAf^{n-1}(x'_{a}-x'_{0a})-2q(q-1)Af^{q-2}(x'_{a}-x'_{0a}),\\
\partial_{nn}\phi&=(n+1)nAf^{n-1}-q(q-1)Af^{q-2}.
\end{align*}
Let $\varepsilon$ be the constant in Lemma \ref{lemma-order 2}. Choose $q$ such that
\begin{align}\label{q}n+1< q< n+1+\min\big\{{1}/{2},\varepsilon\big\}.\end{align}

We first note
$$u\leq v+\phi\quad\text{on }\partial D.$$
This is obviously true on $t=0$. For other parts of $\partial D$, by Lemma \ref{lemma-order 2}, we need to require
\begin{align*}
\frac{C}{\delta^{n-\varepsilon}}= A,\end{align*}
where $C$ is a constant depending only on $n$, $|\varphi|_{C^{n+1}(B_{1}^{+})}$,
$|u|_{L^{\infty}(B_{1}^{+})}$ and
$|H|_{C^{n}(B_{1}^{+})}$.
We also have
\begin{align}\label{star}|D\phi|\le CAf^{n}\leq
C\frac{f^{n}}{\delta^{n-\varepsilon}}\le Cf^{\varepsilon}. \end{align}
Then,
\begin{align}\label{star2}
\begin{split}
1+|D(v+\phi)|^{2}&=1+|Dv|^2+2Dv\cdot D\phi+|D\phi|^2\\
&=1+|Dv|^{2}+O(Af^{n})\\
&=1+c_{1}^{2}+|D_{x'}c_{0}|^{2}+O(x_{n})+O(Af^{n})\\
&=1+c_{1}^{2}+|D_{x'}c_{0}|^{2}+O(f^{\epsilon}).
\end{split}
\end{align}

Next, we write
\begin{equation}\label{eq-Expression-Q}tQ(v+\phi)=tQ(v)+D_{1}+D_{2},\end{equation}
where
$$
D_1=t\bigg(\Delta\phi-\frac{(v_i+\phi_i)(v_j+\phi_j)}{1+|D(v+\phi)^2|}\phi_{ij}
-\bigg[\frac{(v_i+\phi_i)(v_j+\phi_j)}{1+|D(v+\phi)|^2}-\frac{v_iv_j}{1+|Dv|^2}\bigg]v_{ij}\bigg),$$
and
$$
D_2=-n\partial_n\phi+nH\big(\sqrt{1+|D(v+\phi)|^{2}}-\sqrt{1+|Dv|^{2}}\big).$$
By Lemma \ref{lemma-FormalExpansion}, we have
\begin{equation}\label{eq-EstimateQ}|tQ(v)|\le Ct^{n+1}\log t^{-1}\le Ct^{n+\frac{1}{2}}.
\end{equation}
We need to estimate terms in $tQ(v+\phi)$ involving $A$.
We consider $D_1$ first. By \eqref{star}, each $D\phi$ has the order of $Af^n$, which is also bounded by
$f^\varepsilon$ by the choice of $A$. Then,
$$\bigg|\frac{(v_i+\phi_i)(v_j+\phi_j)}{1+|D(v+\phi)|^2}-\frac{v_iv_j}{1+|Dv|^2}\bigg|\le CAf^n.$$
Next, for the second derivatives of $\phi$, we first have
$$|\partial_{an}\phi|+|\partial_{ab}\phi|\le CA(f^n+f^{n-1}|x'|)\le CAf^{n-\frac12}.$$
The coefficient of $\partial_{nn}\phi$ is given by
$$\aligned 1-\frac{(v_n+\phi_n)^2}{1+|D(v+\phi)|^2}&=1
-\frac{c_1^2+O(t)+O(Af^{n})}{1+c_{1}^{2}+|D_{x'}c_{0}|^{2}+O(t)+O(Af^{n})}\\
&=1-H_0^2+O(f^\varepsilon),\endaligned$$
where we used \eqref{eq-Expression_c_1}. Therefore,
$$D_1\le t\big[\big(1-H_0^2+Cf^\varepsilon\big)\big((n+1)nAf^{n-1}-q(q-1)Af^{q-2}\big)+CAf^{n-\frac12}\big],$$
and hence
$$
D_1\le t(1-H_0^2)\big[(n+1)nAf^{n-1}-q(q-1)Af^{q-2}\big]+CAf^{n+\frac12}+CAf^{n+\varepsilon}.
$$
By $t=f-|x'-x_0'|^2$, we have
$$\aligned
D_1&\le (1-H_0^2)\big[(n+1)nAf^{n}-q(q-1)Af^{q-1}\big]+CAf^{n+\frac12}+CAf^{n+\varepsilon}\\
&\qquad-(1-H_0^2)Af^{n-1}\big[n(n+1)-q(q-1)f^{q-n-1}\big]|x'-x_0'|^2.
\endaligned$$
As long as $f$ is small, we obtain
\begin{equation}\label{eq-ExpressionD1}
D_1\le (1-H_0^2)\big[(n+1)nAf^{n}-q(q-1)Af^{q-1}\big]+CAf^{n+\frac12}+CAf^{n+\varepsilon}.
\end{equation}
Next, we discuss $D_{2}.$
By writing \begin{align*}
1+|D(v+\phi)|^{2}
=1+|Dv|^{2}+2Dv\cdot D\phi+|D\phi|^2,
\end{align*}
we have
\begin{align*}
\sqrt{1+|D(v+\phi)|^{2}}=\sqrt{1+|Dv|^2}+\frac{Dv\cdot D\phi+|D\phi|^2/2}{\sqrt{1+|Dv|^2}}
+O\bigg(\frac{(Dv\cdot D\phi+|D\phi|^2)^2}{(1+|Dv|^2)^{3/2}}\bigg).
\end{align*}
Note $D_{x'}\phi=O(Af^n|x'|)$ and $|D\phi|^2=O(Af^{n+\varepsilon})$. Then,
$$\sqrt{1+|D(v+\phi)|^{2}}=\sqrt{1+|Dv|^{2}}+\frac{c_1\partial_n\phi}{\sqrt{1+|Dv|^{2}}}
+O(Af^{n+\varepsilon})+O(Af^{n+\frac12}).$$
Hence,
$$D_2=-n\partial_n\phi\bigg(1-\frac{Hc_1}{\sqrt{1+|Dv|^{2}}}\bigg)+O(Af^{n+\varepsilon})
+O(Af^{n+\frac12}).$$
By $1+|Dv|^2=1+|D_{x'}c_0|^2+c_1^2+O(t)$ and \eqref{eq-Expression_c_1}, we have
$$D_2\le -n\big(1-H_0^2+Ct\big)\big((n+1)Af^{n}-qAf^{q-1}\big)
+CAf^{n+\varepsilon}+CAf^{n+\frac12},$$
and hence
\begin{equation}\label{eq-ExpressionD2}D_2\le -n(1-H_0^2)\big[(n+1)Af^{n}-qAf^{q-1}\big]
+CAf^{n+\varepsilon}+CAf^{n+\frac12}.\end{equation}
By combining \eqref{eq-ExpressionD1} and \eqref{eq-ExpressionD2}, we obtain
$$D_1+D_2\le -q(q-n-1)(1-H_0^2)Af^{q-1}+CAf^{n+\varepsilon}+CAf^{n+\frac12},$$
and hence, with the help of \eqref{eq-Expression-Q} and \eqref{eq-EstimateQ},
$$tQ(v+\phi)\le -q(q-n-1)(1-H_0^2)Af^{q-1}
+CAf^{n+\varepsilon}+CAf^{n+\frac12}.$$
By the choice of $q$ in \eqref{q}, we have, for small $\delta$,
$$Q(v+\phi)\le 0\quad\text{in }D.$$
By the maximum principle, we have
\begin{align*}\begin{split}u\leq v+\phi
&\leq \varphi(x')+c_{1}(x')t+c_{2}(x')t^{2}+\cdots+c_{n}(x')t^{n}
\\&\qquad+c_{n+1,1}(x')t^{n+1}\log t+A(|x'-x'_{0}|^{2}+t)^{n+1}.\end{split}\end{align*}
For the lower bound, we consider $v-\phi$ and, by proceeding similarly, we get
\begin{align*}\begin{split}u&\geq \varphi(x')+c_{1}(x')t+c_{2}(x')t^{2}+\cdot\cdot\cdot+c_{n}(x')t^{n}
\\&\qquad+c_{n+1,1}(x')t^{n+1}\log t-A(|x'-x'_{0}|^{2}+t)^{n+1}.\end{split}\end{align*}
By taking $x'=x'_{0},$ we have the desired result. \end{proof}


\section{The Tangential Regularity}\label{Sec-TangentialSmooth}

In the present section, we study  the regularity along tangential directions.

Let $u\in C(\bar{B}_{1}^{+})\bigcap C^{2}(B_{1}^{+})$
be a solution of
\eqref{eq-Intro-Equation} and \eqref{eq-Intro-BoundaryValue}. Set
\begin{equation}\label{v-Def}v(x',t)=u(x',t)-\varphi(x')-c_{1}(x')t,\end{equation}
where $c_1$ is given by \eqref{eq-Expression_c_1}.
We write \eqref{eq-Intro-Equation} as
\begin{equation}\label{aij}
A_{ij}(Du)u_{ij}-\frac{n}{t}(u_n-H\sqrt{1+|Du|^2})=0\quad\text{in }B_1^+,
\end{equation}
where
$$A_{ij}(p)=\delta_{ij}-\frac{p_ip_j}{1+|p|^2}.$$
 In particular,
\begin{align}\label{v-Eq23}
A_{nn}&=\frac{1+|D_{x'}v+D_{x'}\varphi+D_{x'}c_{1}t|^{2}}
{1+|D_{x'}v+D_{x'}\varphi+D_{x'}c_{1}t|^{2}+(v_{n}+c_{1})^{2}}.\end{align}
These expressions will be needed later on.

It is natural to write an equation for $v$. However, there are drawbacks using
the equation for $v$ directly. We note that the regularity of $v$ is one degree worse than that of
$\varphi$ by the expression of $c_1$ in \eqref{eq-Expression_c_1}. If $\varphi$ is $C^2$, then $v$ is only $C^1$
and hence there is no equation for $v$. Second, even if $\varphi$ is at least $C^3$ and we can write an
equation for $v$, the derivatives of $c_1$ result in expressions with worse regularity in this equation.
In this way, we are unable to get optimal estimate for $u$ or $v$. In the following, we will modify the definition
of $v$ in \eqref{v-Def} and, instead of the function $c_1(x')$, we use its value $c_1(x_0')$ at some point $x_0'$.
Similar methods were used in the proof of Lemma \ref{lemma-order 2} and Lemma \ref{t2}.

We first prove a gradient estimate under the assumptions of Lemma \ref{t2}.

\begin{lemma}\label{lemma-d-order 2}
Assume $\varphi\in C^{2}(B_{1}')$ and $H\in C^1(\bar B_1^+)$, with
$|H|<1$ in $\bar B_1^+$. Let
$ u \in C(\bar{B}_{1}^{+})\bigcap C^{2}(B_{1}^{+})$ be a solution of
\eqref{eq-Intro-Equation}-\eqref{eq-Intro-BoundaryValue} and $v$ be defined as in
\eqref{v-Def}.
Then, there exists a constant $r\in (0,1)$, such that for any $(x',t)\in B_{r}'\times(0,r),$
\begin{align}\label{d2u}|Dv|\leq Ct,\end{align}
where $r$ and $C$ are positive constants
depending only on n, $|\varphi|_{C^{2}(B_{{3}/{4}}')},$
$|u|_{L^{\infty}(B_{1}^{+})}$, and
$|H|_{C^{1}(B_{1}^{+})}$.
Moreover, $u\in C^{1,\beta}(\bar {B}_{r}'\times[0,r]),$ for any $\beta\in (0,1),$ and
$v\in C^{1}(\bar {B}_{r}'\times[0,r]).$ In addition, if $\varphi \in C^{2,{\alpha}}(\bar {B}_{1}')$ and
$H \in C^{1,{\alpha}}(\bar {B}_{1}')$ for some $\alpha\in (0,1)$,
then $v\in C^{1,{\alpha}}(\bar {B}_{r}'\times[0,r]).$
\end{lemma}

\begin{proof} Take any $(x_{0}',t_{0})\in  B_{r}'\times(0,r)$ for some positive $r<\frac{1}{2}\min\{r_1,\delta_1\}$
and set $t_{0}=2\delta.$
Without loss of generality, we assume $\varphi(x'_{0})=0$. Then, by Lemma \ref{t2},
\begin{align}\label{new-1}
|u|\leq|v|+|\varphi-\varphi(x'_{0})|+|\varphi(x'_{0})|+|c_{1}t|\leq C(t+|x'-x'_{0}|),
\end{align}
where $C$ is a positive constant depending only on
$C_{0}$, $|\varphi|_{C^{1}}$ and $|H|_{L^{\infty}}$.
With $e_{n}=(0,1),$
consider the transform $T:B_{1}(e_{n})\rightarrow B_{2r}'\times(0,2r),$ given by
$$x'=x'_{0}+\delta z',\quad t=\delta(s+1).$$
Set
\begin{equation}\label{v-Def0}\widetilde{v}(x',t)=u(x',t)-\varphi(x')-c_1(x_0')t,\end{equation}
and
$$\widetilde{v}^{\delta}(z',s)=\delta^{-1}\widetilde{v}(x',t),\quad u^{\delta}(z',s)=\delta^{-1}u(x',t).$$
We emphasize that $\widetilde v$ here is not the function $v$ defined in \eqref{v-Def}.
Then,
\begin{align}\label{new-4}
|\widetilde{v}^{\delta}|_{L^{\infty}(B_{1}(e_{n}))}\leq \delta^{-1}(Ct^2+C|x'-x'_{0}|t)\le C\delta,
\end{align}
and, by (\ref{new-1}),
\begin{align*}
|u^{\delta}|_{L^{\infty}(B_{1}(e_{n}))}\leq C.
\end{align*}
Here and hereafter, the constant $C$ is independent of $\delta$.
Note, in $ B_{1}(e_{n}),$
\begin{align}\label{new-2}
\Delta{u^{\delta}}-\frac{\partial_iu^{\delta}\partial_ju^{\delta}}{1+|Du^{\delta}|^{2}}\partial_{ij}u^{\delta}
-\frac{n}{s+1}\big[\partial_nu^{\delta}-H(x'_{0}+\delta z',\delta(s+1))\sqrt{1+| Du^{\delta}|^{2}}\big]=0.
\end{align}
This is a mean curvature type equation. The interior gradient estimate implies
\begin{align}\label{continu}
|Du^{\delta}|_{L^{\infty}(B_{7/8}(e_{n}))}\leq C.
\end{align}
(See \cite{GT1983}  or \cite{Han2016Book} for details.)
Then, \eqref{new-2} is uniformly elliptic in $B_{7/8}(e_{n}).$
We now write (\ref{new-2}) as
\begin{align*}
\partial_{i}\bigg(\frac{\partial_iu^{\delta}}{\sqrt{1+|Du^{\delta}|^{2}}}\bigg)
-\frac{n}{(s+1)\sqrt{1+|Du^{\delta}|^{2}}}\partial_nu^{\delta}+\frac{n}{s+1}H(x'_{0}+\delta z',\delta(s+1))=0.
\end{align*}
Fixing $1\le k\le n$ and differentiating with respect to $z_k$, we have
\begin{align*}
\partial_{i}&\bigg(\bigg(\delta_{ij}-\frac{\partial_iu^{\delta}\partial_ju^\delta}{\sqrt{1+|Du^{\delta}|^{2}}}\bigg)
\partial_j(\partial_ku^\delta)\bigg)
-\frac{n}{(s+1)\sqrt{1+|Du^{\delta}|^{2}}}\partial_n(\partial_ku^{\delta})\\
&\qquad+\frac{n\partial_iu^\delta\partial_nu^\delta}{(s+1)(1+|Du^{\delta}|^{2})^{3/2}}\partial_i(\partial_ku^{\delta})
+\frac{n}{s+1}\partial_kH=0.
\end{align*}
Hence, $\partial_ku^\delta$ is a solution of some uniformly elliptic linear equation of divergence form
in $B_{7/8}(e_n)$, with bounded coefficients and bounded nonhomogenuous term by \eqref{continu}.
Therefore, by interior estimates due to de Giorgi and Moser as in \cite{GT1983}, for some
$\alpha\in (0,1)$,
\begin{align}\label{new-3}
|Du^{\delta}|_{C^{\alpha}(B_{3/4}(e_{n}))}\leq C,
\end{align}
where $C$ is a positive constant depending only on
$n$, $|\varphi|_{C^{2}}$, $|H|_{C^{1}}$ and $|Du^{\delta}|_{L^{\infty}(B_{7/8}(e_{n}))}$,
independent of $\delta.$  We can now apply interior $C^{1,\alpha}$-estimates to \eqref{new-2}
and obtain \eqref{new-3} for any $\alpha\in (0,1)$.

A simple calculation yields, for $a\neq n,$
\begin{align}\label{d1}\begin{split}
\partial_a\widetilde{  v}^\delta&=\partial_au^\delta-\partial_a\varphi(x'),\\
\partial_n\widetilde{  v}^\delta&=\partial_nu^\delta-c_1(x_0'),\\
\partial_{ij}\widetilde{  v}^\delta&= \partial_{ij}u^\delta-\delta \partial_{ij}\varphi(x').\end{split}
\end{align}
We now substitute \eqref{d1} in \eqref{new-2}. First, we consider
\begin{align*}
\partial_nu^{\delta}-H(x'_{0}+\delta z',\delta(s+1))\sqrt{1+| Du^{\delta}|^{2}}
= \partial_n\widetilde{  v}^\delta+D_1+D_2+H(x_0',0)D_3,
\end{align*}
where
\begin{align*}
D_1 &=c_1(x_0')-H(x'_{0},0)\sqrt{1+|\partial_n u^{\delta}|^{2}
+|  D_{x'}\widetilde{  v}^{\delta}|^{2}+|  D_{x'}\varphi|^{2}(x_0')+2 D_{x'}\widetilde{  v}^{\delta} D_{x'}\varphi(x_0')},\\
D_2&=\big(H(x'_{0},0)-H(x'_{0}+\delta z',\delta(s+1))\big)\sqrt{1+| Du^{\delta}|^{2}},\\
D_3 &=\sqrt{1+|\partial_n u^{\delta}|^{2}
+|  D_{x'}\widetilde{  v}^{\delta}|^{2}+|  D_{x'}\varphi|^{2}(x_0')+2 D_{x'}\widetilde{  v}^{\delta} D_{x'}\varphi(x_0')}\\
&\qquad-\sqrt{1+|\partial_n u^{\delta}|^{2}
+|  D_{x'}\widetilde{  v}^{\delta}|^{2}+|  D_{x'}\varphi|^{2}(x')+2 D_{x'}\widetilde{  v}^{\delta} D_{x'}\varphi(x')}.
\end{align*}
Set
$$
h=(\partial_n \widetilde{  v}^{\delta})^{2}+2\partial_n\widetilde{  v}^{\delta}c_1(x_0')
+|  D_{x'}\widetilde{  v}^{\delta}|^{2}+2 D_{x'}\widetilde{  v}^{\delta} D_{x'}\varphi(x_0'),$$
and
\begin{align*}
h_a&=-\frac{\partial_a\widetilde{v}+2\partial_a\varphi(x_0')}{2(1+|D_{x'}\varphi|^2(x_0')+c_1^2(x_0'))},\\
h_n&=-\frac{\partial_n\widetilde{v}+2c_1(x_0')}{2(1+|D_{x'}\varphi|^2(x_0')+c_1^2(x_0'))}.\end{align*}
Then,
\begin{align*}
D_1&=c_1(x_0')\bigg(1-\sqrt{1+\frac{h}{1+|D_{x'}\varphi|^2(x_0')+c_1^2(x_0')}}\bigg)\\
&=c_1(x_0')\int_0^1\bigg[1+\frac{sh}{1+|D_{x'}\varphi|^2(x_0')+c_1^2(x_0')}\bigg]^{-1/2}ds\cdot h_i\partial_i\widetilde{v}.
\end{align*}
Note $r$ is small and, by \eqref{d1} and \eqref{continu},
\begin{align*}
1+\frac{sh}{1+|D_{x'}\varphi|^2(x_0')+c_1^2(x_0')}&=(1-s)+s(1+\frac{h}{1+|D_{x'}\varphi|^2(x_0')+c_1^2(x_0')})\\
&\ge(1-s)+s\frac{1+| Du^{\delta}(x',t)|^{2}-C|x'-x_0'|}{1+|D_{x'}\varphi|^2(x_0')+c_1^2(x_0')},
\end{align*}
where $C$ is a positive constant depending only on $|Du^{\delta}|_{L^{\infty}(B_{7/8}(e_{n}))}$, $|\varphi|_{C^{2}}$,
$|H|_{C^{1}}$ and $n.$
By \eqref{new-3}, \eqref{d1} and \eqref{continu}, we obtain
\begin{align*}
c_1(x_0')\int_0^1\big[1+\frac{sh}{1+|D_{x'}\varphi|^2(x_0')+c_1^2(x_0')}\big]^{-1/2}ds\cdot h_i\in C^{\alpha}(B_{3/4}(e_{n})),\end{align*}
and
$$|D_2|\le C\delta,\quad
|D_3|\le C|x'-x_0'|\le C\delta.$$
Hence, \eqref{aij} has the form
\begin{align*}
A_{ij}(Du^{\delta})\partial_{ij}\widetilde{v}^{\delta}+b_{i}^{\delta}(z',s)\partial_i\widetilde{v}^{\delta}+
\delta f^{\delta}(z',s)=0\quad \text{in } B_{{3}/{4}}(e_{n}),
\end{align*}
where $(A_{ij})$ is uniformly elliptic by \eqref{continu} and
$b_{i}^{\delta}$ and $f^{\delta}$ are bounded by
constants depending only on $n$, $|\varphi|_{C^{2}},$ $|u|_{L^{\infty}}$ and
$|H|_{C^{1}}.$
By \eqref{new-4}, \eqref{new-3} and the interior $C^{1,\beta}$-estimates, we get, for any $\beta\in (0,1)$,
\begin{align*}
|\widetilde{v}^{\delta}|_{C^{1,\beta}(B_{1/2}(e_{n}))}
\leq C(|\widetilde{v}^{\delta}|_{L^{\infty}(B_{3/4}(e_{n}))}+\delta|f^{\delta}|_{L^{\infty}(B_{3/4}(e_{n}))})\leq C\delta,
\end{align*}
where $C$ is positive constant independent of $\delta.$
Scaling back, we get for $i=1,\cdots,n,$
$$\bigg|\frac{D\widetilde{v}}{t}(x_{0}',t_{0})\bigg|=\bigg|\frac{D\widetilde{v}^{\delta}}{\delta(s+1)}(e_n)\bigg|\leq C.$$
Hence,  for any $t\in(0,r),$
$$
|D\widetilde{v}(x_{0}',t)|\leq Ct, $$
and
$$[\partial_i\widetilde{v}]_{ C^{\beta}( {B}_{{t}/{4}}(x_0',t))}\le C,$$
where $C$ is a positive constant depending only on $n$, $|u|_{L^{\infty}}$, $|\varphi|_{C^{2}}$
and $|H|_{C^{1}}.$

Recall that $\widetilde v$ is defined in \eqref{v-Def0}.
Next, we claim, for any $\beta \in(0,1),$
$$\widetilde{v}\in C^{1,\beta}(\bar {B}_{r}'\times[0,r]).$$
Take an arbitrarily fixed $(x_1',t)\in{B}_{r}'\times(0,r)$ and set
\begin{align*}
\widehat v(x',t)=u(x',t)-\varphi(x')-c_1(x_1')t.
\end{align*}
Then,
\begin{align*}
\partial_a\widehat v(x',t)&=\partial_a\widetilde{v}(x',t),\\
\partial_n\widehat v(x',t)&=\partial_n\widetilde{v}(x',t)+c_1(x_0')-c_1(x_1').
\end{align*}
Therefore,  for $i=1,\cdots,n,$
\begin{align*}
[\partial_i\widetilde{v}]_{ C^{\beta}( {B}_{{t}/{4}}(x_1',t))}&=[\partial_i\widehat v]_{ C^{\beta}( {B}_{{t}/{4}}(x_1',t))}\le C,\\
|\partial_a\widetilde{v}(x_1',t)|&=|\partial_a\widehat v(x_1',t)|\le Ct,\\
|\partial_n\widetilde{v}(x_1',t)-(c_1(x_1')-c_1(x_0'))|&=|\partial_n\widehat v(x_1',t)|\le Ct,
\end{align*}
where $C$ is independent of $x'_1.$ Since the choice of $x_1'$ is arbitrary, we have,
for any $(x',t)\in {B}_{r}'\times(0,r),$
\begin{align*}
[\partial_i\widetilde{v}]_{ C^{\beta}( {B}_{{t}/{4}}(x',t))}&\le C,\\
|\partial_a\widetilde{v}(x',t)|&\le Ct,\\
|\partial_n\widetilde{v}(x',t)-(c_1(x')-c_1(x_0'))|&\le Ct.
\end{align*}
Note $c_1-c_1(x_0')\in C^1(B_1').$ Then, the claim holds.
Hence, by the definition of $\widetilde{v},$ we have
$u\in C^{1,\beta}(\bar {B}_{r}'\times[0,r]) ,$ for any $\beta\in (0,1).$

Now we compare $Dv$ and $D\widetilde{v}.$ Note
\begin{align*}
\partial_a v&=\partial_a\widetilde{v}-\partial_ac_1(x')t,\\
\partial_nv&=\partial_n\widetilde{v}+c_1(x_0')-c_1(x').
\end{align*}
Therefore, if $\varphi\in C^{2}(B_1')$ and $H\in C^{1}(\bar B_1^+),$ then $v \in C^{1}(\bar {B}_{r}'\times[0,r])$;
if $\varphi\in C^{2,{\alpha}}(B_1')$ and $H\in C^{1,{\alpha}}(\bar B_1^+)$ for some $\alpha\in (0,1)$,
then $v \in C^{1,{\alpha}}(\bar {B}_{r}'\times[0,r])$.
Moreover,
\begin{align*}
|Dv-D\widetilde{v}|\le Ct+C|x'-x_0'|.
\end{align*}
Evaluating at $x'=x_0',$ we have $|Dv(x_0',t)|\le Ct.$ Note that $x_0'$ is chosen arbitrarily.
\end{proof}




We now discuss higher tangential regularity of $v$.

\begin{theorem}\label{lemma-d-order 3}
Assume $\varphi\in C^{2,\alpha}(B_{1}')$ and $H\in C^{1,\alpha}(\bar B_1^+)$, with
$|H|<1$ in $\bar B_1^+$,
for some $\alpha \in(0,1)$. Let
$ u \in C(\bar{B}_{1}^{+})\bigcap C^{2,\alpha}(B_{1}^{+})$ be a solution of
\eqref{eq-Intro-Equation}-\eqref{eq-Intro-BoundaryValue} and $v$ be defined as in
\eqref{v-Def}.
Then, there exists a constant $r\in(0,1)$ such that
\begin{align}\label{t-v3}\frac{v}{t^{2}},\frac{Dv}{t},D^{2}(u-\varphi)
\in C^{\alpha}(\bar{B}_{r}'\times[0,r]),\end{align}
and
\begin{align}\label{t-1-v3}\frac{v}{t},
\frac{v^{2}}{t^{3}},\frac{vv_{t}}{t^{2}},
\frac{v_{t}^{2}}{t}\in C^{1,\alpha}(\bar{B}_{r}'\times[0,r]).\end{align}
\end{theorem}

\begin{proof} Note that (\ref{t-1-v3}) follows from (\ref{t-v3}) easily.
Take any $x_0'\in B_{r}',$ for some small $r,$ and consider in $B_{r}'(x_0')\times(0,r).$
Set
\begin{equation}\label{v-Def1}
\widetilde{v}(x',t)=u(x',t)-\varphi(x')-c_1(x_0')t-D_{x'}c_1(x_0')(x'-x_0')t,\end{equation}
and
\begin{equation}\label{h-Def1}
h(x')=c_1(x')-c_1(x_0')-D_{x'}c_1(x_0')(x'-x_0').
\end{equation}
Then, for $a,b\neq n,$
\begin{align*}
\partial_n\widetilde{v}&=\partial_n u-c_1(x')+h(x'),\\
\partial_a\widetilde{v}&=\partial_au-\partial_a\varphi-\partial_ac_1(x_0')t,\end{align*}
and
\begin{align*}
\partial_{ab}\widetilde{v}&=\partial_{ab}u-\partial_{ab}\varphi,\\
\partial_{an}\widetilde{v}&=\partial_{an}u-\partial_ac_1(x_0'),\\
\partial_{nn} \widetilde{v}&=\partial_{nn}u.
\end{align*}
By a simple substitution in \eqref{aij} and a straightforward calculation, we have
\begin{align}\label{tildav}
A_{ij}\partial_{ij}\widetilde{v}+\frac{b_i}{t}\partial_i\widetilde{v}+f+\widetilde{h}=0,
\end{align}
where, for $a\neq n,$
\begin{align*}
b_a&=ng\big[\partial_a\widetilde{v}+2(\partial_a\varphi+\partial_ac_1(x_0')t)\big],\\
b_n&=-n+ng\big[\partial_n\widetilde{v}+2(c_1(x')-h(x'))\big],\\
f&=A_{ab}\partial _{ab}\varphi+2A_{an}\partial_ac_1(x_0')+ ng
\big[|D_{x'}c_1|^2(x_0')\cdot t+2D_{x'}\varphi (x')D_{x'}c_1(x_0')\big]\\
&\qquad+\frac{n}{t}(H-H_0)\sqrt{1+[\partial_n\widetilde{v}+c_1(x')-h(x')]^2
+|\nabla_{x'}(\widetilde{v}+\varphi)+\nabla_{x'}c_1(x_0')t|^2},\end{align*}
and
\begin{align*}\widetilde{h}&=\frac{n}{t}h(x')\widehat{h},\\
 \widehat{h}&=1+g(h(x')-2c_1(x')),\end{align*}
with
\begin{align*} g&=c_1(x')\int_0^1\frac{1}{2[1+|D_{x'}\varphi|^2+c_1^2(x')]\sqrt{1+sw[1+|D_{x'}\varphi|^2+c_1^2(x')]^{-1}}}ds,\\
w&=\big[|D_{x'}c_1|^2(x_0')\cdot t+2D_{x'}\varphi (x')D_{x'}c_1(x_0')\big]t+h(x')[h(x')-2c_1(x')]\\
&\qquad+\sum_a \partial_a\widetilde{v}\big[\partial_a\widetilde{v}+2(\partial_a\varphi+\partial_ac_1(x_0')t)\big]
+\partial_n\widetilde{v}\big[\partial_n\widetilde{v}+2(c_1(x')-h(x'))\big].
\end{align*}
The definition of $h$ in \eqref{h-Def1} implies
$h\in C^{1,\alpha}(\bar B_{1/2}')$
and
$$|h(x')|\le C|x'-x_0'|^{1+\alpha}.$$
By Lemma \ref{lemma-d-order 2},
we have $A_{ij},b_i,f,\widehat{h}\in C^\alpha(\bar B_{r_{0}}'\times[0,r_{0}]),$ for some small $r_0$.
Here we point out $\frac{n}{t}(H-H_0)\in C^\alpha(\bar B_{r_{0}}'\times[0,r_{0}])$ can be easily derived from $H\in C^{1,\alpha}(\bar{B}_1^+).$
Hence,
\begin{align}\label{h/t}\begin{split}
|b_n(x',0)-b_n(x_0',0)|&\le C|x'-x_0'|^\alpha,\\
|\widetilde{h}(x',t)|&\le \frac{C}{t}|x'-x_0'|^{1+\alpha}.\end{split}
\end{align}
By Lemma \ref{lemma-d-order 2},  $Dv=0$ on $\{t=0\}$. Hence, with \eqref{v-Eq23}, we have
\begin{align}\label{p3}
\begin{split}
A_{nn}(x',0)&=1-H_0^{2}(x'),\\
b_n(x_0',0)&=-n(1-H_0^{2}(x_0')).
\end{split}
\end{align}
Then, for sufficiently small constants $r$ and $\varepsilon_0$, we have, in
$B_{r}'\times(0,r)$,
\begin{align}\label{negative3}
\begin{split}
(1+\alpha)A_{nn}+b_n&\leq -b_{0},\\
3A_{nn}+b_n\bigg(1+\frac{1-\varepsilon_0}{1+\alpha}\bigg)& \leq-b_{0},
\end{split}
\end{align}
for some positive constant $b_{0}$.

We claim, for some $c_{2} \in C^{\alpha}(\bar{B}'_{r})$
and any $(x',t) \in B_{r}'\times(0,r) $,
\begin{align}\label{t2alpha3}
|v(x',t)-c_{2}(x')t^{2}| \leq Ct^{2+\alpha}.
\end{align}
The expression of $c_{2}$ will be given in the proof below.
We point out that since $c_{2}$ is only $C^{\alpha},$ we cannot differentiate $c_{2}.$
Set
\begin{align*}
L(w)&=A_{ij}\partial_{ij}w+\frac{b_i}{t}\partial_{i}w,\\
Q(w)&=L(w)+f+\widetilde{h},
\end{align*}
where $A_{ij}, b_i,f$ and $\widetilde{h}$ are evaluated at $x$ and $Dv$.
By \eqref{tildav}, we have $Q(\widetilde{v})=0.$ Without loss of generality, we assume $x_0'=0.$
Set
\begin{align*}
\psi=\mu_{1}t(\mu_3|x'|^{2}+t^{2})^{\frac{\alpha+1}{2}}+\mu_{2}t^{2+\alpha}, \end{align*}
and
\begin{align*}
\overline{ v}  =c_{2}(0)t^{2}+\psi,
\end{align*}
where $\mu_{1}$, $\mu_{2}$, $\mu_3$ and $c_2(0)$ are constants to be determined. We will
choose $\mu_1$ and $\mu_2$ large
and $\mu_{3}$  small. Then,
\begin{align}\label{eq-equation-v-overline}
Q( \overline{ v})=L(\psi)+f+\widetilde{h}+2(A_{nn}+b_n)c_{2}(0).
\end{align}
Set
\begin{align}\label{c2-03}
c_{2}(0)=-\frac{f}{2(A_{nn}+b_n)}(0).
\end{align}
A straightforward calculation yields
\begin{align*}
L(\psi)&=\mu_{1}t(\mu_3|x'|^{2}+t^{2})^{\frac{\alpha-1}{2}}(\alpha+1)
\bigg[3A_{nn}+b_n\bigg(1+\frac{1}{\alpha+1}\frac{\mu_3|x'|^{2}+t^{2}}{t^2}\bigg)\\
&\qquad+\mu_3(2A_{an}+b_a)\frac{x_{a}}{t}+\mu_3\delta_{ab}+\frac{\alpha-1}{\mu_3|x'|^{2}
+t^{2}}[A_{nn}t^{2}+2\mu_3A_{an}x_{a}t+\mu_3^2A_{ab}x_{a}x_{b}]\bigg]\\
&\qquad+\mu_{2}t^{\alpha}[(\alpha+1)(\alpha+2)A_{nn}+(\alpha+2)b_n].
\end{align*}
In the above expression, we write
$$\frac{1}{\alpha+1}=\frac{1-\varepsilon_0}{\alpha+1}+\frac{\varepsilon_0}{\alpha+1},$$
for some $\varepsilon_0$ in \eqref{negative3}.
Note that
$$|A_{nn}(x',t)|\le 1,$$
and all $A_{ij}$ and $b_i$ are bounded. By \eqref{h/t}, \eqref{p3}, \eqref{negative3} and the
Cauchy inequality, we have
\begin{align*}L(\psi)&\le\mu_{1}t(\mu_3|x'|^{2}+t^{2})^{\frac{\alpha-1}{2}}(\alpha+1)
\bigg[-b_0+C\sqrt{\mu_3}+\frac{(1-\alpha)t^2}{\mu_3|x'|^{2}
+t^{2}}\bigg] \\
&\qquad-b_0\mu_{2}t^{\alpha} -a_0\mu_{1}\frac{1}{t}(\mu_3|x'|^{2}+t^{2})^{\frac{\alpha+1}{2}},\end{align*}
where $a_0$ is a positive constant independent of $\mu_3$.
We first fix $\mu_3$ such that $C\sqrt{\mu_3}\le b_0/4$.
Next, take $M>0$ to be determined. For any $(x',t)\in B_{r}'\times(0,r)$
with $|x'| \geq Mt,$ we have
$$L(\psi)\le \mu_{1}t(\mu_3|x'|^{2}+t^{2})^{\frac{\alpha-1}{2}}(\alpha+1)
\bigg[-\frac34b_0+\frac{1-\alpha}{\mu_3M^{2}
+1}\bigg]-\frac{1}{t}a_0\mu_{1}(\mu_3|x'|^{2}+t^{2})^{\frac{\alpha+1}{2}}.$$
Hence, for $M$ sufficiently large, we have, for any $(x',t)\in B_{r}'\times(0,r)$
with $|x'| \geq Mt,$
\begin{align}\label{mu1.13}
L(\psi)\le -\frac{b_0}{2}\mu_{1}t(\mu_3|x'|^{2}+t^{2})^{\frac{\alpha-1}{2}}(\alpha+1)
-\frac{1}{t}a_0\mu_{1}\mu_3^{\frac{\alpha+1}{2}}(|x'|^{2}+t^{2})^{\frac{\alpha+1}{2}}.\end{align}
On the other hand, we have, for any $(x',t)\in B_{r}'\times(0,r)$
with $|x'| \le Mt,$
$$L(\psi)\le \mu_{1}t^\alpha(\mu_3M^{2}+1)^{\frac{\alpha-1}{2}}(\alpha+1)(1-\alpha)
-b_0\mu_{2}t^{\alpha}-\frac{1}{t}a_0\mu_{1}(\mu_3|x'|^{2}+t^{2})^{\frac{\alpha+1}{2}},$$
Hence, for each $\mu_1$, we can choose $\mu_2$ as a big multiple of $\mu_1$ and obtain,
for any $(x',t)\in B_{r}'\times(0,r)$
with $|x'| \le Mt,$
\begin{align}\label{mu1.23}L(\psi)\le -Bb_0\mu_{1}
(|x'|^{2}+t^{2})^{\frac{\alpha}{2}}-\frac{1}{t}a_0\mu_{1}\mu_3^{\frac{\alpha+1}{2}}(|x'|^{2}+t^{2})^{\frac{\alpha+1}{2}},\end{align}
for some positive constant $B$.
By combining \eqref{mu1.13} and \eqref{mu1.23},  we obtain
\begin{align}\label{mu3}
L(\psi)\leq -C_*\mu_{1}t(|x'|^{2}+t^{2})^{\frac{\alpha-1}{2}}
-\frac{1}{t}a_0\mu_{1}\mu_3^{\frac{\alpha+1}2}(|x'|^{2}+t^{2})^{\frac{\alpha+1}{2}},
\end{align}
and, with \eqref{eq-equation-v-overline}
and \eqref{c2-03},
\begin{align*}
Q(\overline{v})&\leq -C_*\mu_{1}t(|x'|^{2}+t^{2})^{\frac{\alpha-1}{2}}
-a_0\mu_{1}\mu_3^{\frac{\alpha+1}{2}}\frac{(|x'|^{2}+t^{2})^{\frac{\alpha+1}{2}}}{t}
+C(|x'|^\alpha+t^\alpha)+
C\frac{|x'|^{1+\alpha}}{t}\\&\leq 0,
\end{align*}
by choosing $\mu_{1}\geq C/(a_0\mu_3^{\frac{\alpha+1}{2}})$.
Next, note
\begin{align*}
\widetilde{v}-c_2(0)t^2\leq v+C|x'|^{\alpha+1}t+|c_2(0)|t^2
\le Ct^2+C|x'|^{\alpha+1}t.
\end{align*}
In order to have $\widetilde{v}\leq \overline{v}$ on $\partial(B_{r}'\times(0,r)),$ we take
\begin{align*}
\mu_1\ge \left(C+\frac{C}{r^{\alpha}}\right)\mu_{3}^{-\frac{\alpha+1}{2}}.
\end{align*}
In summary, we have $Q(\overline{v})\le Q(\widetilde{v})$ in $B_{r}'\times(0,r)$ and
$\widetilde{v}\leq \overline{v}$ on $\partial(B_{r}'\times(0,r))$.
By the maximum principle, we get $\widetilde{v}\leq \overline{v}$ in $B_{r}'\times(0,r),$ and hence
$$\widetilde{v}\leq c_{2}(0)t^{2}+\psi\quad \text{in } B_{r}'\times(0,r).$$
Similarly, we have
$$\widetilde{v}\geq c_{2}(0)t^{2}-\psi\quad \text{in } B_{r}'\times(0,r).$$
By taking $x'= 0,$ we have (\ref{t2alpha3}) for $x'= 0.$
We can prove (\ref{t2alpha3}) for any $(x',t)\in B_{r}'\times(0,r)$ by a similar method.
Instead of (\ref{c2-03}), we set
\begin{align*}
c_{2}(x')=-\frac{f}{2(A_{nn}+b_n)}(x',0).
\end{align*}
Note $c_2\in C^{\alpha}(B_{r}').$ 

With \eqref{t2alpha3}, we will prove \eqref{t-v3}.
We take any $(x_{0}',t_{0})\in  B_{{r}/{2}}'\times(0,{r}/{2})$ and
set $t_{0}=2\delta.$ With $e_{n}=(0,1),$ consider the transform
$T:B_{1}(e_{n})\rightarrow B_{r}'\times(0,r)$ given by
$$x'=x'_{0}+\delta z',\quad t=\delta(s+1).$$
Set
$$\widetilde{v}^{\delta}(z',s)=\delta^{-2}\big[u(x',t)-\varphi(x')-c_1(x_0')t-D_{x'}c_1(x_0')(x'-x_0')t-c_{2}(x_{0}')t^{2}\big].$$
Then, by (\ref{t2alpha3}) and \eqref{tildav},
\begin{align}\label{vr1}
|\widetilde{v}^{\delta}|_{L^{\infty}(\bar{B}_{1}(e_{n}))}\le\delta^{-2}C\big[t^{2+\alpha}+|x'-x_0'|^{1+\alpha}t
+|x'-x_0'|^{\alpha}t^2\big]
\le C\delta^\alpha,
\end{align}
and
\begin{align}\label{vdel}
A_{ij}\partial_{ij}\widetilde{v}^{\delta}+\frac{b_i}{s+1}\partial_{i}\widetilde{v}^{\delta}
+g(x'_{0}+\delta z',\delta(s+1))+\widetilde{h}&=0,
\end{align}
where
$$g=f+2(A_{nn}+b_n)c_{2}(x_{0}').$$ Note
$$|g(x'_{0}+\delta z',\delta(s+1))|_{C^{\alpha}(\bar{B}_{1}(e_{n}))}\leq
C\delta^\alpha,$$
and by \eqref{h/t},
\begin{align*}
|\widetilde{h}|_{L^{\infty}(\bar{B}_{1}(e_{n}))}\le C\delta^\alpha,
\end{align*}
Since $h\in C^{1,\alpha},$ we can check easily
\begin{align}\label{h/t1}
\left[\frac{h}{t}\right]_{C^{\alpha}(\bar{B}_{1}(e_{n}))}\le C\delta^\alpha,\quad
[\widetilde{h}]_{C^{\alpha}(\bar{B}_{1}(e_{n}))}\le C\delta^\alpha.
\end{align}
By the Schauder estimate, we get $\widetilde{v}^{\delta}\in C^{2,\alpha}(\bar{B}_{1/2}(e_{n}))$ and
\begin{align*}|\widetilde{v}^{\delta}|_{C^{2,\alpha}(\bar{B}_{1/2}(e_{n}))}
&\leq C\left\{|\widetilde{v}^{\delta}|_{L^{\infty}(\bar{B}_{1}(e_{n}))}+|g(x'_{0}
+\delta z',\delta(s+1))|_{C^{\alpha}(\bar{B}_{1}(e_{n}))}+|\widetilde{h}|_{C^{\alpha}(\bar{B}_{1}(e_{n}))}\right\}\\
&\leq C \delta^{\alpha}.\end{align*}
Now we can scale back and note that $x_0'\in B_{{r}/{2}}'$ is arbitrary.
By \eqref{h/t1} and the definition of $h$,  we obtain, for any $(x',t)\in B_{{r}/{2}}'\times(0,{r}/{2}),$
\begin{align*}
D^2(u-\varphi), \frac1tD\big(u-\varphi-c_1(x')t\big),\frac1{t^2}\big(u-\varphi-c_1(x')t\big)\in C^{\alpha}(B_{t/4}(x',t)),
\end{align*}
and
\begin{align*}
&|\partial_{nn}(u-\varphi)-2c_2(x')|+|\partial_{ab}(u-\varphi)|+|\partial_{an}(u-\varphi)-\partial_a c_1(x')|\le Ct^{\alpha},\\
&\bigg|\frac1t\partial_a\big(u-\varphi-c_1(x')t\big)\bigg|+\bigg|\frac1t\partial_n\big(u-\varphi-c_1(x')t\big)-2c_2(x')\bigg|\le Ct^{\alpha},
    \\&\bigg|\frac1{t^2}\big(u-\varphi-c_1(x')t-c_2(x')t^2\big)\bigg|\le Ct^{\alpha}.
\end{align*}
Hence, we get \eqref{t-v3}.
\end{proof}

We now prove a general result.

\begin{theorem}\label{lemma-d-order l}
Assume $\varphi\in C^{l,\alpha}(B_{1}')$ and $H\in C^{l-1, \alpha}(\bar B_1^+)$,
with $|H|<1$ in $\bar B_1^+$,
for some $l\geq 3$ and $\alpha \in(0,1)$. Let
$ u \in C(\bar{B}_{1}^{+})\bigcap C^{l,\alpha}(B_{1}^{+})$ be a solution of
\eqref{eq-Intro-Equation} and \eqref{eq-Intro-BoundaryValue}  and $v$ be defined as in
\eqref{v-Def}.
Then, there exists a constant $r\in(0,1)$ such that, for $j=0,1,\cdots,l-2,$
\begin{align}\label{t-v}\frac{D^{j}_{x'}v}{t^{2}},\frac{D^{j}_{x'}Dv}{t},D^{j}_{x'}D^{2}(u-\varphi)
\in C^{\alpha}(\bar{B}_{r}'\times[0,r]),\end{align}
and
\begin{align}\label{t-1-v}\frac{D^{j}_{x'}v}{t},
\frac{D^{j}_{x'}(v^{2})}{t^{3}},\frac{D^{j}_{x'}(vv_{t})}{t^{2}},
\frac{D^{j}_{x'}(v_{t}^{2})}{t}\in C^{1,\alpha}(\bar{B}_{r}'\times[0,r]).\end{align}
\end{theorem}

\begin{proof}
Note that (\ref{t-1-v}) follows from (\ref{t-v}) easily and we only need to prove (\ref{t-v}).
For some small $r,$ we set, for any $x_0'\in B_r',$
\begin{align*}
\widetilde{v}&=u(x',t)-\varphi(x')-c_1(x_0')t-\sum_{|\beta|=1}^{l-1}
\frac{1}{\beta!}\partial^\beta_{x'} c_1(x_0')(x'-x_0')^\beta t.
\end{align*}
Similarly as for \eqref{tildav}, we can derive the equation for $\widetilde{v}$
\begin{align}\label{tildav1}
A_{ij}\partial_{ij}\widetilde{v}+\frac{b_i}{t}\partial_i\widetilde{v}+f+\widetilde{h}=0.
\end{align}
We will prove (\ref{t-v}) by an induction on $l$.
The case $l=2$ follows from (\ref{t-v3}).
For $l\geq 3$, we fix an integer $1\leq k\leq l-2$ and assume (\ref{t-v}) holds for $j=0,1,\cdots,k-1.$
Then, consider the case $j=k.$ By differentiating \eqref{tildav1} by $D_{x'}^k$, we have
$$A_{ij}D_{x'}^k\partial_{ij}\widetilde{v}+\frac{b_i}{t}D_{x'}^k\partial_{i}\widetilde{v}+f+\widetilde{h}=0, $$
with $A_{ij},b_i,f,\widetilde{h}$ satisfying the same properties needed in Lemma \ref{lemma-d-order 3}.
We claim, for some $c_{k,2}\in C^{\alpha}(\bar{B}'_{r})$
and any $(x',t) \in B_{r}'\times(0,r) $,
\begin{align*}
|D_{x'}^{k}v(x',t)-c_{k,2}(x')t^{2}| \leq Ct^{2+\alpha}.
\end{align*} The proof is similar as that of \eqref{t2alpha3} and is omitted.
Then, \eqref{t-v} for $j=k$ follows similarly.
\end{proof}

We point out that Tonegawa \cite{Tonegawa1996MathZ} already
proved the tangential regularity of $v$.
The present form is used in the expansions to be discussed in the next section.

\section{Regularity Along The Normal Direction}\label{sec-N-Regularityies}

In this section, we discuss the regularity along the normal direction.

First, with Theorem \ref{lemma-d-order l}, we rewrite \eqref{aij} as
\begin{align}\label{v-Eq}
A_{nn}(Du)\partial_{nn}v+2A_{an}(Du)\partial_{an}u+A_{ab}(Du)\partial_{ab}u-\frac{n}{t}\partial_nv+N=0\quad\text{in }B_1^+,
\end{align}
where
\begin{align*}N&=\frac{n}{t}\big[H\sqrt{1+|D_{x'}u|^{2}+(\partial_nu)^{2}}
-c_{1}\big].\end{align*}
We write
$$N=N_1+N_2,$$
where
\begin{align*}
N_{1}(x',t)&=\frac{n}{t}(H-H_{0})\sqrt{1+|D_{x'}u|^{2}+(\partial_{n}u)^{2}},\\
N_{2}(x',t)&=\frac{n}{t}\big[H_{0}\sqrt{1+|D_{x'}u|^{2}+(\partial_{n}u)^{2}}-c_{1}\big].
\end{align*}
We note, by \eqref{eq-Expression_c_1},
\begin{align*}
N_2=\frac{n}{t}c_{1}\left(\sqrt{1+\frac{h}{1+|D_{x'}\varphi|^{2}+c_{1}^{2}}}-1\right),\end{align*}
where
\begin{align*}
h=|D_{x'}v+D_{x'}c_{1}t|^{2}+2D_{x'}\varphi D_{x'}(v+c_{1}t)+(\partial_{n}v)^{2}+2\partial_{n}vc_{1}.
\end{align*}
Hence,  we have
$$\frac{h}{1+|D_{x'}\varphi|^2+c_1^2}=2h_i\partial_iv+2h_0t,$$
where
$$\aligned h_a&=\frac{2\partial_a\varphi+\partial_ac_1t+\partial_av/2}{1+|D_{x'}\varphi|^2+c_1^2},\\
h_n&=\frac{c_1+\partial_nv/2}{1+|D_{x'}\varphi|^2+c_1^2},\\
h_0&=\frac{D_{x'}\varphi\cdot D_{x'}c_{1}
+|D_{x'}c_{1}|^2t/{2}}{1+|D_{x'}\varphi|^2+c_1^2}.\endaligned$$
Therefore, we can express $N_2$ 
by
$$\aligned
N_2&=\frac{nc_1}{t}\int_0^1\frac{d\ }{ds}\big[1+\frac{sh}{1+|D_{x'}\varphi|^2+c_1^2}\big]^{1/2}ds\\
&=\frac{nc_1}{t}\int_0^1\bigg[1+\frac{sh}{1+|D_{x'}\varphi|^2+c_1^2}\bigg]^{-1/2}ds\cdot (h_i\partial_iv+h_0t).
\endaligned$$
With such an $N_2$, we can write \eqref{v-Eq} as
\begin{align}\label{vr3}
A_{nn}\partial_{nn}v+2A_{an}\partial_{an}u+A_{ab}\partial_{ab}u+\frac{b_i}{t}\partial_{i}v+f=0,
\end{align}
where $A_{ij}, b_i$ and $f$ are smooth in $x$, $H$, $D_{x'}\varphi$,$D_{x'}H$,$\frac{H-H_{0}}{t}$, $D_{x'}^2\varphi$ and
$Dv$. In particular,
\begin{align}\label{54}\begin{split}
A_{nn}&=\frac{1+|D_{x'}v+D_{x'}\varphi+D_{x'}c_{1}t|^{2}}
{1+|D_{x'}v+D_{x'}\varphi+D_{x'}c_{1}t|^{2}+(\partial_{n}v+c_{1})^{2}},\\
b_n&=-n+\frac{n(c_1^2+c_1\partial_{n}v/2)}{1+|D_{x'}\varphi|^2+c_1^2}
\int_0^1\bigg[1+\frac{sh}{1+|D_{x'}\varphi|^2+c_1^2}\bigg]^{-1/2}ds.\end{split}\end{align}
Then,
$$\partial_{nn}v+\frac{b_n}{tA_{nn}}\partial_nv
=\frac{1}{A_{nn}}\bigg[-2A_{an}\partial_{an}u-A_{ab}\partial_{ab}u-b_a\frac{\partial_av}{t}-f\bigg].$$
By \eqref{d2u}, we have
${b_n}/{A_{nn}}=-n$ on $\{t=0\}$.
Hence,
\begin{equation}\label{eq-p1}\partial_{nn}v-\frac{n}{t}\partial_nv=F,\end{equation}
where
$$F=\frac{1}{A_{nn}}\bigg[-2A_{an}\partial_{an}u-A_{ab}\partial_{ab}u-b_a\frac{\partial_av}{t}-f\bigg]
-\frac{1}{tA_{nn}}(b_n+nA_{nn})\partial_nv.$$
By \eqref{d2u}, we have $(b_n+nA_{nn})(\cdot, 0)=0$.
In fact, a straightforward calculation, with the help of  \eqref{54}, yields
$$b_n+nA_{nn}=h_i\partial_iv+h_0t,$$
where $h_1, \cdots, h_n$ and $h_0$ are smooth functions in
$x$, $H$ $D_{x'}\varphi$, $D_{x'}H$, $D_{x'}^2\varphi$ and $Dv$. Hence,
$$(b_n+nA_{nn})\frac{\partial_nv}{t}=h_a\partial_nv\cdot\frac{\partial_av}{t}+h_n\cdot\frac{v_t^2}{t}+h_0\partial_nv,$$
where the summation for $a$ is from 1 to $n-1$.
Therefore,
\begin{equation}\label{eq-p2}F=2\widetilde A_{an}\partial_{an}u+\widetilde A_{ab}\partial_{ab}u
+\widetilde b_a\frac{\partial_av}{t}+\widetilde b_n\frac{v_t^2}{t}+\widetilde f,\end{equation}
where $\widetilde A_{an}$, $\widetilde A_{ab}$,
$\widetilde b_a$, $\widetilde b_n$ and $\widetilde f$ are smooth functions
in $x$, $H$, $D_{x'}\varphi$,$D_{x'}H$,$\frac{H-H_{0}}{t}$, $D_{x'}^2\varphi$ and
$Dv$.
We note that
$F$ is a smooth function in $t$ and
$$
v_t, \frac{v_t^2}{t}, D_{x'}v_t,
D^2_{x'}u,$$
and that $F$ depends on $x'$ through derivatives of $\varphi$ and $H$ up to order $2$ and $1$, respectively.
Moreover,
$$F\text{ is linear in } \frac{v_t^2}{t}.$$
Then, with Theorem \ref{lemma-d-order l}, the method for regularity along the normal direction in \cite{HanJiang} can be applied to draw the conclusion.

\end{document}